\documentclass[11pt]{amsart}
\usepackage[dvipsnames]{xcolor}
\usepackage{amsmath, amscd, amssymb, mathrsfs}
\usepackage[all]{xy}
\usepackage{setspace}

\usepackage[]{todonotes} 

\usepackage{marginnote}  
\usepackage{enumitem} 
\usepackage[colorlinks=true]{hyperref}  
\usepackage[right,mathlines,running]{lineno}
\usepackage[abbrev]{amsrefs}

\newtheorem{thm}{Theorem}[section]
\newtheorem{cor}[thm]{Corollary}
\newtheorem{lem}[thm]{Lemma}
\newtheorem{prop}[thm]{Proposition}
\theoremstyle{definition}
\newtheorem{definition}[thm]{Definition}
\theoremstyle{remark}

\newtheorem{rem}[thm]{Remark}
\numberwithin{equation}{section}

\newcommand{\ai}{\sqrt{-1}} 
\newcommand{\isom}{\stackrel{\sim}{\to}} 

\newcommand{\rk}{\mathrm{rk}}
\newcommand{\tr}{\mathrm{tr}}

\newcommand{\Id}{\mathrm{Id}}

\newcommand{\surj}{\to\kern-1.8ex\to}

\newcommand{\vertiii}[1]{{\left\vert\kern-0.25ex\left\vert\kern-0.25ex\left\vert #1 
    \right\vert\kern-0.25ex\right\vert\kern-0.25ex\right\vert}}
\allowdisplaybreaks[4]

\begin{document}
\title[Quot-scheme limit and its applications]{Quot-scheme limit of Fubini--Study metrics and its applications to balanced metrics} 
\author{Yoshinori Hashimoto} 
\address{Department of Mathematics, Tokyo Institute of Technology, 2-12-1 Ookayama, Meguro-ku, Tokyo, 152-8551, Japan.}
\email{hashimoto@math.titech.ac.jp}

\author{Julien Keller}
\address{D\'epartement de Math\'ematiques, Universit\'e du Qu\'ebec \`a Montr\'eal (UQ\`AM), C.P. 8888, Succ. Centre-Ville, Montr\'eal (Qu\'ebec) H3C 3P8 Canada.}
\email{julien.keller@math.cnrs.fr}

\begin{abstract}
	We present some results that complement our prequels \cite{HK1,HK2} on holomorphic vector bundles. We apply the method of the Quot-scheme limit of Fubini--Study metrics developed therein to provide a generalisation to the singular case of the result originally obtained by X.W.~Wang for the smooth case, which states that the existence of balanced metrics is equivalent to the Gieseker stability of the vector bundle. We also prove that the Bergman 1-parameter subgroups form subgeodesics in the space of hermitian metrics. This paper also contains a review of techniques developed in \cite{HK1,HK2} and how they correspond to their counterparts developed in the study of the Yau--Tian--Donaldson conjecture.
\end{abstract}


\maketitle

\section{Introduction}

The theorem due to Donaldson \cite{Do-1983,Don1985,Don1987} and Uhlenbeck--Yau \cite{U-Y} states that a holomorphic vector bundle over a smooth complex projective variety admits a Hermitian--Einstein metric if the bundle is slope stable. Together with the theorem by Kobayashi \cite{Kobookjp} and L\"ubke \cite{Luebke}, it follows that the vector bundle admits a Hermitian--Einstein metric if and only if it is slope stable. This is an important theorem in complex geometry that provides an important link between differential and algebraic geometry, which was proved by using deep analytic results in \cite{Do-1983,Don1985,Don1987,U-Y}.

In the papers \cite{HK1,HK2} we discussed attempts at establishing a more direct link between the Hermitian--Einstein metrics and slope stability, from the point of view of computing the asymptotic slope of the appropriate energy functional that arises in the variational formulation of the problem. Following these works, the purpose of this paper is twofold: one is to prove some results (Sections \ref{scbg1pss} and \ref{scgstbm}) that complement the prequels \cite{HK1,HK2}, and the other is to provide a summary (Section \ref{scrmrhk1}) of what is developed in \cite{HK1,HK2} and compare it with the analogous ideas in the study of Yau--Tian--Donaldson conjecture, i.e.~the case of varieties (Section \ref{sccpcytd}). More detailed content of this paper is summarised below.

\bigskip

\noindent {\bf Organisation of the paper.} In Section \ref{scrmrhk1}, we survey the methods developed in \cite{HK1} and the results proved in \cite{HK1,HK2}. We then present how these methods can be regarded as a vector bundle version of the ideas proposed in the study of constant scalar curvature K\"ahler and K\"ahler--Einstein metrics in Section \ref{sccpcytd}; we also give brief comments on the Deligne pairing in Section \ref{sccmdlp}, which plays an important role in the previous work by Boucksom--Hisamoto--Jonsson \cite{BHJ2} and Phong--Ross--Sturm \cite{PRS} for the case of varieties but not in our papers \cite{HK1,HK2}, by discussing its relationship to the Bott--Chern class. After showing that the Bergman 1-parameter subgroups do indeed define subgeodesics in Section \ref{scbg1pss}, as expected from the case of varieties, we prove in Section \ref{scgstbm} that the method of the Quot-scheme limit of Fubini--Study metrics provides a variational characterisation of the Gieseker stability, generalising the results originally proved by X.W.~Wang \cite{Wang1} and later by Phong--Sturm \cite{P-S03} to the case of torsion-free sheaves over a $\mathbb{Q}$-Gorenstein log terminal projective variety. Finally in Section \ref{Section-effective}, we provide some effective results and numerical aspects of the Donaldson functional.

\medskip

\noindent {\bf Acknowledgements.} Both authors thank Yuji Odaka for helpful discussions. The first author is partially supported by JSPS KAKENHI (Grant-in-Aid for Early-Career Scientists), Grant Number JP19K14524. The second author is supported by an NSERC Discovery Grant.

\tableofcontents

\noindent \textbf{Notation.} Throughout in this paper (except for Remarks \ref{rmsgvar}, \ref{rmsgdsgvr}, and Section \ref{scgstbm}), $X$ stands for a smooth complex projective variety, and $\mathcal{E}$ stands for a holomorphic vector bundle of rank $r$ over $X$, and we write $\mathcal{E} (k)$ for $\mathcal{E} \otimes \mathcal{O}_X (k)$ with a very ample line bundle $\mathcal{O}_X (1)$; we shall write $\mathcal{O}_X(k)$ for $\mathcal{O}_X (1)^{\otimes k}$ ($k \in \mathbb{Z}$). We assume, just in order to simplify the exposition, that $\mathcal{E}$ does not split into a direct sum of holomorphic subbundles (i.e.~$\mathcal{E}$ is irreducible), noting that the reducible case can be treated similarly by considering each irreducible component.

We fix a K\"ahler metric $\omega$ on $X$ once and for all, in the K\"ahler class $c_1 (L)$ where $L = \mathcal{O}_X (1)$. We shall also write $\mathrm{Vol}_L$ for the volume $\int_X \omega^n/n!$ of $X$ with respect to $\omega$.

In Section \ref{scgstbm}, we treat the case when $X$ is a singular variety, and its singular locus is denoted by $\mathrm{Sing} (X)$. Likewise, the singular locus of a coherent sheaf $\mathcal{E}$ is denoted by $\mathrm{Sing} (\mathcal{E})$. In this paper, the \textit{regular locus} $X^{reg} \subset X$ denotes a further subset of $X \setminus (\mathrm{Sing} (X) \cup \mathrm{Sing} (\mathcal{E}))$ as given in Definition \ref{dfxreg}. We may also write $\mathrm{Sing} (X , \mathcal{E})$ for $\mathrm{Sing} (X) \cup \mathrm{Sing} (\mathcal{E})$.

\section{Summary of the methods developed in \cite{HK1,HK2}} \label{scrmrhk1}

\subsection{Fubini--Study metrics}

We start by recalling the classical Fubini--Study metrics for vector bundles. The reference is \cite{Wang1} or \cite[Chapter 5]{M-M}.

Since $\mathcal{O}_X(1)$ is very ample, the vector bundle $\mathcal{E} (k)$ is globally generated for all large enough $k$, so that the sheaf surjection
\begin{equation} \label{dfrhoss}
	\rho : H^0 (X , \mathcal{E} (k)) \otimes \mathcal{O}_X (-k) \to \mathcal{E} ,
\end{equation}
defined by evaluation at each point, is surjective. In what follows, to streamline the exposition, we shall further assume $k > \mathrm{reg} (\mathcal{E})$, where $\mathrm{reg} (\mathcal{E})$ is an integer called the \textbf{Castelnuovo--Mumford regularity} defined by
\begin{equation*}
	\mathrm{reg} (\mathcal{E}) := \inf_{k \in \mathbb{Z}} \{ \mathcal{E} \text{ is $k$-regular.}  \}
\end{equation*}
where we recall that $\mathcal{E}$ is $k$-regular if $H^i (X , \mathcal{E}  (k-i)) = 0$ for all $i > 0$ (the existence of the regularity number is justified by Serre vanishing theorem).

The sheaf surjection (\ref{dfrhoss}) implies that we have a holomorphic map
\begin{equation} \label{dfgrembkdr}
	\Phi : X \to \mathrm{Gr} (r , H^0 (X , \mathcal{E} (k))^{\vee})
\end{equation}
to the Grassmannian of $r$-planes in $H^0 (X , \mathcal{E} (k))^{\vee}$, such that the universal bundle $\mathcal{U}$ over the Grassmannian (i.e.~the dual of the tautological bundle) is pulled back by $\Phi$ to $\mathcal{E} (k)$.

A positive definite hermitian form $H$ on the vector space $H^0 (X , \mathcal{E} (k))^{\vee}$ naturally defines a hermitian metric on the universal bundle $\mathcal{U}$ over the Grassmannian. Pulling this back by $\Phi$ we get a hermitian metric on $\mathcal{E} (k)$; fixing a reference hermitian metric on $\mathcal{O}_X(1)$, this means that we get a hermitian metric on $\mathcal{E}$. The metric thus constructed is called the \textbf{$k$-th Fubini--Study metric} on $\mathcal{E}$ defined by the hermitian form $H$; note that this precisely agrees with the classical Fubini--Study metrics induced by the embedding to the projective space when $\mathcal{E}$ is a line bundle.

The construction above can also be described as follows. Note first that any positive definite hermitian form $H^0 (X , \mathcal{E} (k))^{\vee}$ can be written, up to an overall constant multiple which does not play an important role in this paper, as $\sigma^* \sigma$ for some $\sigma \in SL(H^0 (X , \mathcal{E} (k))^{\vee})$. Then, as pointed out by X.W.~Wang \cite[Remark 3.5]{Wang1}, there exists a $C^{\infty}$-map
\begin{equation} \label{dfsmmpq}
	Q : \mathcal{E} \to H^0 (X , \mathcal{E} (k)) \otimes C^{\infty}_X (-k),
\end{equation}
where $C^{\infty}_X (-k)$ is the sheaf of smooth sections of $\mathcal{O}_X (-k)$, such that the $k$-th Fubini--Study metric $h_{\sigma}$ defined by $\sigma^* \sigma$ can be written as 
\begin{equation} \label{dffsmsgm}
	h_{\sigma} = Q^* \sigma^* \sigma Q,
\end{equation}
where $Q^*$ is the formal adjoint of $Q$ with respect to some fixed reference hermitian metrics. More details can be found in \cite[Remark 3.5]{Wang1} or \cite[Section 1.3]{HK1}; note further that, although they make use of an $L^2$-orthonormal basis for $H^0 (X , \mathcal{E} (k))$, this can be replaced by any basis as proved in \cite[Theorem 5.1.16]{M-M}.

We write $\mathcal{H}_k$ for the set of $k$-th Fubini--Study metrics. Noting that the map $SL(H^0 (X , \mathcal{E} (k))^{\vee}) \ni \sigma \mapsto h_{\sigma} = Q^* \sigma^* \sigma Q \in \mathcal{H}_k$ factors through $SL(H^0 (X , \mathcal{E} (k))^{\vee}) / SU (N_k)$ where $N_k = \dim H^0 (X , \mathcal{E} (k))$, we find that $\mathcal{H}_k$ is parametrised by the homogeneous manifold $SL(H^0 (X , \mathcal{E} (k))^{\vee}) / SU (N_k)$.

\subsection{Quot-scheme limit of Fubini--Study metrics} \label{revqslimfs}
We recall some key concepts from \cite{HK1} that we need in what follows, which the reader is referred to for more details on this section. 

Recalling the description of the Fubini--Study metric $h_{\sigma}$ as in (\ref{dffsmsgm}) by using the map $Q$ as defined in (\ref{dfsmmpq}), we naturally get a family $\{ h_{\sigma_t} \}_{t \ge 0}$ of Fubini--Study metrics defined by a 1-parameter subgroup (1-PS) $\{ \sigma_t \}_{t \ge 0} \subset SL(H^0 (X , \mathcal{E} (k))^{\vee})$ as 
\begin{equation*}
	h_{\sigma_t} := Q^* \sigma_t^* \sigma_t Q.
\end{equation*}
Assuming as we may that $\sigma_t$ is generated by $\zeta \in \mathfrak{sl} (H^0 (X , \mathcal{E}(k))^{\vee})$ as $\sigma_t = e^{\zeta t}$, we call the above $\{ h_{\sigma_t} \}_{t \ge 0}$ the \textbf{Bergman 1-PS} generated by $\zeta \in \mathfrak{sl} (H^0 (X , \mathcal{E}(k))^{\vee})$; further, when $\zeta$ has rational eigenvalues, it is called the \textbf{rational Bergman 1-PS}.

The main technique developed in \cite{HK1} is to evaluate the limit of $h_{\sigma_t}$ as $t \to +\infty$ for $\zeta \in \mathfrak{sl} (H^0 (X, \mathcal{E}(k))^{\vee})$ with rational eigenvalues, in terms of the Quot-scheme limit. We give a quick summary of it below.

Suppose that $\zeta \in \mathfrak{sl} (H^0 (X,\mathcal{E}(k))^{\vee})$ has eigenvalues $w_1 , \dots , w_{\nu} \in \mathbb{Q}$, with the ordering
\begin{equation} \label{ordlambda}
	w_1 > \cdots > w_{\nu} .
\end{equation}
We consider the action of $\zeta$ on $H^0(X,\mathcal{E}(k))$ which is not the natural dual action, but the one that is natural with respect to certain metric duals (see \cite[(2.6)]{HK1} and the discussion that follows). In any case, such an action gives us the weight decomposition
\begin{equation*}
 	H^0 (X, \mathcal{E}(k)) = \bigoplus_{i = 1}^{\nu} V_{-w_i , k}
 \end{equation*}
where $\zeta$ acts on $V_{-w_i , k}$ via the $\mathbb{C}^*$-action $T: \mathbb{C}^* \curvearrowright V_{-w_i , k}$ defined by $T \mapsto T^{-w_i}$ (cf.~\cite[Section 2.1]{HK1}); here we introduced an auxiliary variable $T$ by $T := e^{-t}$, so that the limit $t \to + \infty$ corresponds to $T \to 0$. The above decomposition naturally leads to the filtration
\begin{equation}\label{VT'}
V_{ \le -w_i , k} := \bigoplus_{j= 1}^{i} V_{-w_j, k},
\end{equation}
of $H^0 (X, \mathcal{E}(k)) $ by its vector subspaces.

Recalling the sheaf surjection (\ref{dfrhoss}), the filtration \eqref{VT'} also gives rise to the one
\begin{equation} \label{filtevs}
	0 \neq \mathcal{E}_{\le -w_{1}} \subset \cdots \subset \mathcal{E}_{\le -w_{\nu}} = \mathcal{E}
\end{equation}
of $\mathcal{E}$ by subsheaves, where $\mathcal{E}_{\le - w_i}$ is a coherent subsheaf of $\mathcal{E}$ defined by the quotient map
\begin{equation*}
	\rho_{\le - w_i} : V_{\le - w_i ,k} \otimes \mathcal{O}_X(-k) \to \mathcal{E}_{\le - w_i}
\end{equation*}
induced from $\rho$ as defined in (\ref{dfrhoss}). As in \cite[Lemma 2.5]{HK1}, we can modify this filtration on a Zariski closed subset of $X$, to get a filtration
\begin{equation} \label{satfiltevs}
	0 \neq \mathcal{E}'_{\le -w_{1}} \subset \cdots \subset \mathcal{E}'_{\le - w_{\nu}} = \mathcal{E}
\end{equation}
of $\mathcal{E}$ by saturated subsheaves. We say that a filtration is trivial if it is equal to $0 \subsetneq \mathcal{E}$.

\begin{rem} \label{rmetrlev}
	When the eigenvalues $w_1 , \dots , w_{\nu}$ of $\zeta$ are only real, as opposed to rational, exactly the same argument applies so as to get the filtration (\ref{satfiltevs}), with the only difference being that the grading of the filtration is given by real numbers.
\end{rem}

In describing the limit of the Bergman 1-PS $\{ h_{\sigma_t} \}_{t \ge 0}$, only a certain subset of the subsheaves in (\ref{satfiltevs}) matters, in the sense that we only need to consider the subsheaves in (\ref{satfiltevs}) such that the associated graded sheaf has a nontrivial rank. More precisely, following \cite[Definition 2.2]{HK1}, we can pick a certain subset
\begin{equation} \label{defsubswgt}
\{ w_{\alpha} \}_{\alpha =\hat{1}}^{\hat{\nu}} \subset \{ w_i \}_{i=1}^{\nu}	
\end{equation}
with $\{ \hat{1} , \dots , \hat{\nu} \} \subset \{ 1 , \dots , \nu \}$, by means of the Quot-scheme limit as explained below; what (\ref{defsubswgt}) precisely means is that the subscript $\alpha$ runs over a subset $\{ \hat{1} , \dots , \hat{\nu} \}$ of $\{ 1 , \dots , \nu \}$, with the ordering given by $\hat{1} < \hat{2} < \dots < \hat{\nu}$.

We recall the quotient map (\ref{dfrhoss})
\begin{equation*}
	\rho : H^0(X, \mathcal{E}(k)) \otimes \mathcal{O}_X (-k) \to \mathcal{E}
\end{equation*}
for $\mathcal{E}$, and note that its $\mathbb{C}^*$-orbit defined by $\zeta \in \mathfrak{sl} (H^0 (X,\mathcal{E}(k))^{\vee})$ can be written as
\begin{equation*}
	\rho_T := \rho \circ T^{\zeta} : H^0(X, \mathcal{E}(k)) \otimes \mathcal{O}_X (-k) \to \mathcal{E}.
\end{equation*}
We then consider a coherent sheaf defined as
\begin{equation*}
	\hat{\rho} : H^0(X, \mathcal{E}(k)) \otimes \mathcal{O}_X (-k) \to \bigoplus_{i=1}^{\nu} \mathcal{E}_{-w_i}
\end{equation*}
with
\begin{equation} \label{dfewigr}
	\mathcal{E}_{-w_i} : = \mathcal{E}_{\le -w_i} / \mathcal{E}_{\le -w_{i-1}}.
\end{equation}
It is well-known that $\hat{\rho}$ defined as above is equal to limit of $\rho_T$ in the Quot-scheme under the $\mathbb{C}^*$-action $T^{\zeta}$ \cite[Lemma 4.4.3]{H-L}. The subset $\{ \hat{1} , \dots , \hat{\nu} \}$ in (\ref{defsubswgt}) consists of the indices $i$ such that $\mathrm{rk} (\mathcal{E}_{-w_i}) >0$  (see also \cite[Definition 2.2]{HK1}). It turns out that $\hat{1} = 1$ (see \cite[Remark 2.3]{HK1}), and the reader is referred to \cite[Section 2]{HK1} for more details.


We shall often argue over a certain Zariski open subset of $X$ as defined below.

\begin{definition} \label{dfxreg}
	We define $X^{reg}$ be the Zariski open subset of $X$ over which the sheaves $\mathcal{E}_{\le -w_i}$ in (\ref{filtevs}) and $\mathcal{E} / \mathcal{E}_{\le -w_i}$ are all locally free \cite[Definition 2.4]{HK1}, and such that each $\mathcal{E}_{\le -w_i}$ agrees with the saturated subsheaf $\mathcal{E}'_{\le -w_i}$ in (\ref{satfiltevs}).
\end{definition}

For each $\alpha \in \{ \hat{1} , \dots , \hat{\nu} \}$ in (\ref{defsubswgt}), the quotient sheaf $\mathcal{E}_{-w_{\alpha}} = \mathcal{E}_{\le -w_{\alpha}} / \mathcal{E}_{\le -w_{\alpha -1}}$ is locally free over $X^{reg}$ (and agrees as a $C^{\infty}$-vector bundle with the quotient vector bundle of $\mathcal{E}_{\le -w_{\alpha}}$ by $\mathcal{E}_{\le -w_{\alpha -1}}$; see \cite[discussion following Definition 2.4]{HK1}).

The definition (\ref{dfewigr}) of $\mathcal{E}_{-w_i}$, combined with the above definition of $X^{reg}$, we get a $C^{\infty}$-isomorphism 
\begin{equation} \label{mdualisoev}
\mathcal{E} \stackrel{\sim}{\to}  \bigoplus_{\alpha =\hat{1}}^{\hat{\nu}} \mathcal{E}_{-w_{\alpha}}
\end{equation}
of smooth complex vector bundles over $X^{reg}$ \cite[(2.8)]{HK1}. Moreover, we have a gauge transformation on $\bigoplus_{\alpha =\hat{1}}^{\hat{\nu}} \mathcal{E}_{-w_{\alpha}}$ over $X^{reg}$ by the constant endomorphism
\begin{equation} \label{defewt}
	e^{wt} := \mathrm{diag} (e^{ w_{\hat{1}} t} , \cdots , e^{w_{\hat{\nu}} t}),
\end{equation}
with $e^{w_{\alpha}t }$ acting on the factor $\mathcal{E}_{-w_{\alpha}}$.

An important observation is that the map $Q^*$, defined as the formal adjoint of $Q$ (\ref{dfsmmpq}) as
\begin{equation*}
	Q^* :  \overline{H^0 (X , \mathcal{E} (k))^{\vee}} \otimes \overline{C^{\infty}_X (-k)^{\vee}} \to \overline{\mathcal{E}^{\vee}},
\end{equation*}
can be regarded as a $C^{\infty}$-version of the quotient map $\rho$ (\ref{dfrhoss}), up to taking the metric dual in the domain and the range \cite[Lemma 1.22]{HK1}. This seems to suggest that the limit of the Bergman 1-PS $\{ h_{\sigma_t} \}_{t \ge 0}$ as $t \to + \infty$ can be related to the Quot-scheme limit $\bigoplus_{\alpha =\hat{1}}^{\hat{\nu}} \mathcal{E}_{-w_{\alpha}}$, up to the metric duality isomorphism $\mathcal{E} \isom \overline{\mathcal{E}^{\vee}}$ (as a $C^{\infty}$-vector bundle), and indeed it is the main technical result that was established in \cite{HK1}. More precisely, we define a hermitian metric on $\mathcal{E} |_{X^{reg}}$ by
\begin{equation} \label{defrn1ps}
		\hat{h}_{\sigma_t} :=  e^{- w t} h_{\sigma_t}  e^{- w t} , 
\end{equation}
with (\ref{mdualisoev}) understood, which we call the \textbf{renormalised Bergman 1-PS} associated to $\sigma_t$ \cite[Definition 2.9]{HK1}. An important fact is that this 1-PS is convergent in $C^{\infty}_{\mathrm{loc}}$ over $X^{reg}$ \cite[Proposition 2.8]{HK1}, and we call the limit
\begin{equation} \label{eqrnqsl}
		\hat{h} :=   \lim_{t \to + \infty} e^{- w t} h_{\sigma_t}  e^{- w t}
\end{equation}
the \textbf{renormalised Quot-scheme limit} of $h_{\sigma_t}$ \cite[Definition 2.9]{HK1}, which is positive definite over $X^{reg}$ \cite[Lemma 2.11]{HK1}.

\begin{prop} \emph{(see \cite[Proposition 2.8 and Lemma 2.11]{HK1})}
	The renormalised Bergman 1-PS converges in $C^{\infty}_{\mathrm{loc}}$ over $X^{reg}$ as $t \to + \infty$, and its limit defines a well-defined hermitian metric on $\mathcal{E}$ via (\ref{mdualisoev}) over $X^{reg}$.
\end{prop}

The above limit (\ref{eqrnqsl}) is only defined on a Zariski open subset $X^{reg}$ of $X$, and may well be degenerate on $X \setminus X^{reg}$. In spite of this drawback, we can evaluate the degeneracy of $\hat{h}$ by comparing $Q^*$ with $\rho$ and by using the resolution of singularities. The argument is technical, but can be carried out by using the methods in Jacob \cite{Jacob} and Sibley \cite{Sibley}, which occupies a large portion of the technical argument in \cite[Section 3]{HK1}.

\begin{rem} \label{rmopnml1}
	Throughout in what follows, we shall assume that the operator norm (i.e. the modulus of the maximum eigenvalue) of $\zeta$ is at most 1, as pointed out in \cite[Remark 2.1]{HK1}.
\end{rem}

\begin{rem} \label{rmsgvar}
	Inspection of \cite{HK1} reveals that almost all the arguments so far carry over word-by-word to the case when $X$ is a (not necessarily smooth) projective variety and when $\mathcal{E}$ is a torsion-free sheaf, by considering hermitian metrics on $X \setminus ( \mathrm{Sing} (X) \cup \mathrm{Sing} (\mathcal{E}))$ instead of $X$; the only exception is that the map $\Phi$ in (\ref{dfgrembkdr}) needs to be replaced by a rational map that is well-defined on $X \setminus \mathrm{Sing} (\mathcal{E})$. This means that $X^{reg}$ as in Definition \ref{dfxreg} should be replaced by the Zariski open subset in $X$ excluding all the singular sets of sheaves $\mathcal{E}_{\le -w_i}$ and $\mathcal{E} / \mathcal{E}_{\le -w_i}$ and that of the background data, i.e.~$\mathrm{Sing} (X) \cup \mathrm{Sing} (\mathcal{E})$. In the arguments later (except for Remark \ref{rmsgdsgvr} and Section \ref{scgstbm}), however, it will be important that $X$ is smooth.
\end{rem}

\subsection{The non-Archimedean Donaldson functional}

We recall an important functional defined by Donaldson \cite{Don1985}. Let $\mathcal{H}_{\infty}$ be the set of all smooth hermitian metrics on $\mathcal{E}$.

\begin{definition} \label{dfdnfc}
 Given two hermitian metrics $h_0$ and $h_1$ on $\mathcal{E}$, \textbf{the Donaldson functional} ${\mathcal M}^{Don} : \mathcal{H}_{\infty}\times \mathcal{H}_{\infty} \to \mathbb{R}$ is defined as $${\mathcal M}^{Don}(h_1,h_0) := {\mathcal M}_1^{Don}(h_1,h_0) - \mu(\mathcal{E}){\mathcal M}_2^{Don}(h_1,h_0),$$
 where $${\mathcal M}_1^{Don}(h_1,h_0) :=\int_0^1 \int_X \tr \left( h_t^{-1} \partial_t h_t \cdot F_t\right) \frac{\omega^{n-1}}{(n-1)!} dt$$ and $${\mathcal M}_2^{Don}(h_1,h_0) :=\frac{1}{\mathrm{Vol}_L}\int_X \log\det(h_0^{-1}h_1)\frac{\omega^n}{n!},$$ with
 \begin{equation*}
 	\mu (\mathcal{E}) := \frac{\int_X c_1 (\mathcal{E}) \wedge \omega^{n-1} / (n-1)!}{\mathrm{rk} (\mathcal{E})} .
 \end{equation*}
 
 In the above, $\{ h_t \}_{0 \le t \le 1} \subset \mathcal{H}_{\infty}$ is a smooth path of hermitian metrics between $h_0$ to $h_1$ and $F_t$ denotes ($\ai / 2 \pi$) times the Chern curvature of $h_t$, with respect to the holomorphic structure of $\mathcal{E}$.
\end{definition}

\begin{rem} \label{remordargdfn}
	Throughout in what follows, we shall fix the second argument of ${\mathcal M}^{Don}$ as a reference metric. Thus ${\mathcal M}^{Don}(h,h_0)$ is regarded as a function of $h$ with a fixed reference metric $h_0$.
\end{rem}

The critical point of the Donaldson is the Hermitian--Einstein metric. An important property of the Donaldson functional is that it is convex along geodesics in $\mathcal{H}_{\infty}$, with an appropriate notion of geodesics in $\mathcal{H}_{\infty}$; see \cite{Don1985,Kobook} for more details. Thus, the existence of the critical point of the Donaldson functional can be, at least conceptually, characterised by the positivity of its asymptotic slope. The main result of \cite{HK1} is the explicit description of the asymptotic slope of the Donaldson functional along the rational Bergman 1-PS by means of the algebro-geometric data.

We now recall the non-Archimedean Donaldson functional from \cite{HK1}, defined for a rational Bergman 1-PS generated by $\zeta \in \mathfrak{sl} (H^0 (X , \mathcal{E} (k)))$. Writing $w_1 , \dots , w_{\nu} \in \mathbb{Q}$ for the weights of $\zeta$ as in (\ref{ordlambda}), we choose $j(\zeta , k) \in \mathbb{N}$ to be the minimum integer so that 
\begin{equation} \label{defsmjzk}
	j(\zeta , k) w_i \in \mathbb{Z}
\end{equation}
for all $i=1 , \dots, \nu$. Writing $\bar{w}_i:= j(\zeta , k) w_i$, we may replace the filtration (\ref{satfiltevs}) by
\begin{equation} \label{itsatfiltevs}
	0 \neq \mathcal{E}'_{\le - \bar{w}_{1}} \subset \cdots \subset \mathcal{E}'_{\le - \bar{w}_{\nu}} = \mathcal{E}
\end{equation}
which is graded by integers. With this understood, the following was defined in \cite[Definition 4.3]{HK1}.

\begin{definition} \label{dfnadfn}
	\textbf{The non-Archimedean Donaldson functional} $\mathcal{M}^{\mathrm{NA}} (\zeta , k)$ is a rational number defined for $\zeta \in \mathfrak{sl} (H^0 (X , \mathcal{E} (k)))$ with rational eigenvalues as
	\begin{equation*}
		\mathcal{M}^{\mathrm{NA}} (\zeta , k) := \frac{2}{j(\zeta , k)} \sum_{q \in \mathbb{Z}} \mathrm{rk} (\mathcal{E}'_{\le q}) \left(  \mu(\mathcal{E}) - \mu (\mathcal{E}'_{\le q}) \right).
	\end{equation*}
\end{definition}

We recall, in the above, that the \textbf{slope} of a (necessarily torsion-free) subsheaf $\mathcal{F} \subset \mathcal{E}$ on $X$ is defined by
	\begin{equation*}
		\mu ( \mathcal{F} ) := \frac{\mathrm{deg} (\mathcal{F})}{\mathrm{rk} (\mathcal{F})},
	\end{equation*}
	where, by noting that $\det \mathcal{F}:= \left( \bigwedge^{\mathrm{rk} (\mathcal{F})} \mathcal{F} \right)^{\vee \vee}$ is a holomorphic line bundle on $X$ \cite[Proposition 5.6.10]{Kobook}, the degree is given by $\mathrm{deg} (\mathcal{F}) := \int_X c_1 (\det \mathcal{F}) \wedge c_1 (\mathcal{O}_X(1))^{n-1} / (n-1)!$.

We can also define \cite[Definition 4.1]{HK2} a rational number $J^{\mathrm{NA}} ( \zeta  , k)$ for $\zeta \in \mathfrak{sl} (H^0 (X , \mathcal{E} (k)))$ with rational eigenvalues by
\begin{equation} \label{eqjnavb}
	J^{\mathrm{NA}} (\zeta, k) := \max_{\alpha , \beta \in \{ \hat{1} , \dots , \hat{\nu} \}} |w_{\alpha} - w_{\beta}|
\end{equation}
where we recall (\ref{defsubswgt}) for the definition of $\hat{1} , \dots , \hat{\nu}$. We can easily show \cite[Remark 4.2]{HK2} that $J^{\mathrm{NA}} (\zeta, k) = 0$ if and only if the corresponding filtration (\ref{satfiltevs}) is trivial, i.e.~equals $0 \subsetneq \mathcal{E}$.

An elementary yet important fact is that the positivity of $\mathcal{M}^{\mathrm{NA}} (\zeta , k)$ is equivalent to the slope stability of $\mathcal{E}$, as stated in the following.

\begin{prop} \emph{(see \cite[Proposition 6.2]{HK1})} \label{prop62hk1}
	The non-Archimedean Donaldson functional $\mathcal{M}^{\mathrm{NA}} (\zeta , k)$ is positive (resp.~nonnegative) for all $k \ge \mathrm{reg}(\mathcal{E})$ and all $\zeta \in \mathfrak{sl} (H^0 (X , \mathcal{E}(k))^{\vee})$, with rational eigenvalues, whose associated filtration (\ref{satfiltevs}) is nontrivial, if and only if $\mathcal{E}$ is slope stable (resp.~semistable), i.e.~for any subsheaf $\mathcal{F}$ of $\mathcal{E}$ with $0 < \mathrm{rk} (\mathcal{F}) < \mathrm{rk} (\mathcal{E})$ we have
	\begin{equation*}
		\mu(\mathcal{E}) > \mu (\mathcal{F}), \quad (\text{resp.~} \mu(\mathcal{E}) \ge \mu (\mathcal{F})).
	\end{equation*}
\end{prop}

This can be proved by an argument similar to the proof of Proposition \ref{res1} which is presented later (see also \cite[Section 5]{HK1}).

\begin{rem}
	When $\mathcal{E}$ splits into a direct sum of holomorphic vector subbundles, the right notion to consider is the slope polystability: $\mathcal{E} = \bigoplus_{1 \le l \le m} \mathcal{E}_l$ is said to be slope polystable if each direct summand $\mathcal{E}_l$ is slope stable and $\mu ( \mathcal{E}_{l_1} ) = \mu ( \mathcal{E}_{l_2} )$ for all $1 \le l_1 , l_2 \le m$.
\end{rem}

\subsection{Summary of results in \cite{HK1,HK2}}

We now recall the main results of \cite{HK1} as follows.

\begin{thm} \emph{(see \cite[Theorems 1, 2, and Corollary 6.3]{HK1})} \label{thmlnsdf}
There exists a constant $c_k >0$ that depends only on the reference metric $h_{\mathrm{ref}} \in \mathcal{H}_{\infty}$ and $k \in \mathbb{N}$ such that
\begin{equation}\label{coercive}
		\mathcal{M}^{Don}(h_{\sigma_t} ,  h_{\mathrm{ref}}) \ge  \mathcal{M}^{\mathrm{NA}} (\zeta , k) t - c_{k}
\end{equation}
holds for all $t \ge 0$ and all hermitian $\zeta \in \mathfrak{sl}(H^0(\mathcal{E}(k))^{\vee})$ with rational eigenvalues. We can further show that
\begin{equation} \label{rmlnsdf}
\mathcal{M}^{Don} (h_{\sigma_t}, h_{\mathrm{ref}}) = \mathcal{M}^{\mathrm{NA}} (\zeta , k) t + O(1),
\end{equation}
where $O(1)$ stands for the term that remains bounded as $t \to + \infty$, i.e $\mathcal{M}^{Don}$ has log norm singularities along any rational Bergman 1-PS. In particular, we have
\begin{equation*}
\lim_{t \to + \infty} \frac{\mathcal{M}^{Don} (h_{\sigma_t}, h_{\mathrm{ref}})}{t} = \mathcal{M}^{\mathrm{NA}} (\zeta , k),
\end{equation*}
which shows that $\mathcal{M}^{\mathrm{NA}} (\zeta , k)$ is the term that controls the asymptotic behaviour of $\mathcal{M}^{Don}(h_{\sigma_t}, h_{\mathrm{ref}})$.
\end{thm}


Another way of stating the property \eqref{coercive} is to say that $\mathcal{M}^{Don}$ is \textit{coercive} (resp. bounded from below) along rational Bergman 1-PS if $\mathcal{E}$ is slope stable (resp. slope semistable).

The corollary of the above result is that we can show that the existence of the Hermitian--Einstein metrics implies $\mathcal{E}$ is slope stable \cite[Section 7]{HK1}. It may also be worth noting that the analysis used to prove the above results is elementary, cf.~\cite[Section 3]{HK1}.

The reverse direction of the correspondence was discussed in \cite{HK2}. Writing $\mathcal{H}_k$ for the set of $k$-th Fubini--Study metrics, we have a natural inclusion $\mathcal{H}_k \subset \mathcal{H}_{\infty}$ for each $k$. While $\mathcal{H}_k$ is a ``small'' subset of $\mathcal{H}_{\infty}$ parametrised by a finite dimensional homogeneous manifold $SL(H^0 (X , \mathcal{E} (k))^{\vee}) / SU (N_k)$, it turns out that the union of $\mathcal{H}_k$'s is dense in $\mathcal{H}_{\infty}$ with respect to the $C^p$-topology (for any fixed $p \in \mathbb{N}$), i.e.~
\begin{equation} \label{eqbexds}
	\mathcal{H}_{\infty} = \overline{\bigcup_{k \in \mathbb{N}} \mathcal{H}_k}.
\end{equation}
This fact follows from a foundational result in K\"ahler geometry, called the \textbf{asymptotic expansion of the Bergman kernel}, which is also called the Tian--Yau--Zelditch expansion, and is essentially a theorem in analysis. We do not give a detailed account of this result, and refer the reader to \cite{Catlin}, \cite{Wang2}; note that several proofs have been written especially when $\mathcal{E}$ has rank one (see e.g.~the book \cite{M-M} and references therein).  An elementary proof can be found in \cite{BBS2008}.


The main result of \cite{HK2} is that, \textit{if} we assume that a uniform version of Theorem \ref{thmlnsdf} holds, we can prove that the slope stability implies Hermitian--Einstein metrics by only using elementary analysis except for the asymptotic expansion of the Bergman kernel. The precise statement is as follows.

\begin{thm} \emph{(see \cite[Theorem 1]{HK2})} 
	Suppose that the estimate in the theorem above holds uniformly in $k$, i.e.~
	\begin{equation*}
		\mathcal{M}^{Don}(h_{\sigma_t} ,  h_{\mathrm{ref}}) \ge  \mathcal{M}^{\mathrm{NA}} (\zeta , k) t - c_{\mathrm{ref}}
\end{equation*}
for a constant $c_{\mathrm{ref}} > 0$ that depends only on the reference metric. Then we can prove that the stability implies the existence of the Hermitian--Einstein metric by using only elementary analytic methods except for $\mathcal{H}_{\infty} = \overline{\bigcup_{k \in \mathbb{N}} \mathcal{H}_k}$ in (\ref{eqbexds}) which is a consequence of the asymptotic expansion of the Bergman kernel.
\end{thm}

\section{Comparison to the case of the Yau--Tian--Donaldson conjecture} \label{sccpcytd}

\subsection{Dictionary between vector bundles and manifolds}

The methods and results summarised above are motivated by the recent progress on the Yau--Tian--Donaldson conjecture, surveyed e.g.~in \cite{Boucksom-icm,Th,Eys1,Dem1}, and it seems reasonable to have a table of correspondence between the vector bundles case and the varieties case. Indeed, our approach in \cite{HK1} can be regarded as a vector bundle version of the results concerning the Yau--Tian--Donaldson conjecture as established e.g.~in \cite{Boucksom-icm,BBJ,BHJ1,BHJ2,Paul,PRS}. For example, one of our main results Theorem \ref{thmlnsdf} (or rather its consequence (\ref{rmlnsdf})) can be regarded as a vector bundle version of a result by Boucksom--Hisamoto--Jonsson \cite{BHJ2}, Paul \cite{Paul}, and Phong--Ross--Sturm \cite{PRS} (amongst many other related results).

It is well-known that the role played by the Mabuchi energy in the case of varieties is almost exactly the same as that of the Donaldson functional in the case of vector bundles; the critical point of these functionals are precisely the canonical metrics, and both of them are convex along geodesics in the space of metrics (although the convexity for the Mabuchi energy is a much more subtle issue due to the weaker regularity of the geodesics in the space of K\"ahler potentials \cite{B-B,Chen2000}). It is also well-known that the maximally destabilising subsheaf for vector bundles corresponds to the optimal destabilising test configuration. We list below (table \ref{dictionary}) how the objects reviewed in Section \ref{revqslimfs} correspond to the ones in the case of varieties, i.e.~study of constant scalar curvature K\"ahler and K\"ahler--Einstein metrics.

It is well-known in the case of varieties that a test configuration defines a subgeodesic in the space of K\"ahler metrics (see e.g.~\cite{Ber16,BHJ2}). In Section \ref{scbg1pss} we provide a vector bundle version of this result in Corollary \ref{crb1pssgd}.

Another important topic in the study of constant scalar curvature K\"ahler and K\"ahler--Einstein metrics is what is known as Donaldson's quantisation, which can be regarded as a finite dimensional approximation of the canonical metric by a sequence of balanced metrics \cite{Do-2001}. The vector bundle version of this result was established by X.W.~Wang \cite{Wang2}. He also proved that the existence of the balanced metrics is equivalent to the Gieseker stability of the vector bundle \cite{Wang1}, where we note that the analogous result for the varieties case is due to Luo \cite{Luo} and Zhang \cite{Zhang}.

In Section \ref{scgstbm}, we apply the method of the Quot-scheme limit of Fubini--Study metrics that we reviewed in Section \ref{scrmrhk1} to give a generalisation of this result. The dictionary also extends to the balancing flow for manifolds defined in \cite{Do-2001,Fine2010} and for the bundle version in \cite{K-S}. The first one  provides a quantisation of the Calabi flow while the second one provides a quantisation of the Yang--Mills flow.

On a Fano manifold (without nontrivial holomorphic vector field), the behaviour of Mabuchi energy restricted to Fubini--Study metrics of level $k_0$ (for a certain $k_0$ sufficiently large) is sufficient to test $K$-stability. Actually, building on the partial $C^0$ estimate from Sz\'ekelyhidi \cite{Sz2016} solving a conjecture of Tian and the work Paul \cite{Paul} about CM-stability, Boucksom--Hisamoto--Jonsson \cite{BHJ2} get that coercivity of the Mabuchi energy on the space of positive metrics implies uniform $K$-stability and thus the existence of a K\"ahler--Einstein metric by Chen--Donaldson--Sun. Moreover, as explained in the discussion after \cite[Theorem 2.9]{Odaka} or \cite[page 3]{LWC}, in order to test $K$-stability it is sufficient to work with 1-PS degenerations in a fixed projective space (induced by the space of holomorphic section of a fixed power of the anticanonical bundle). This is actually a consequence of Chen--Donaldson--Sun too. Then the fact that it is sufficient to consider a fixed $k_0$  is a consequence of \cite[Theorem C]{BHJ2} (more precisely $(i) \Rightarrow (ii)$ which can be obtained from Theorem A and the equivalence $(ii) \Leftrightarrow (iii)$). Eventually, Section \ref{Section-effective} is addressing the counterpart of this result for the Mabuchi energy to the bundle case for the Donaldson functional.

\begin{rem}
	We have not yet found an appropriate analogue of the $J$-functional for vector bundles, while the quantity  $J^{\mathrm{NA}} (\zeta , k)$ defined in (\ref{eqjnavb}) does seem to play a role analogous to the non-Archimedean $J$-functional.
\end{rem}


\begin{center}
\begin{table}[h!]
  \begin{tabular}{p{6.1cm}|p{5.9cm}}
    \begin{center} {\it Vector bundles} \end{center} &  \begin{center} {\it Manifolds} \end{center}\\
    \hline
     Filtration $\mathcal{E}'_{\le -w_{1}} \subset \cdots \subset \mathcal{E}'_{\le - w_{\nu}}$ (\ref{satfiltevs}) of $\mathcal{E}$ & Test configuration \cite{Do-2002,Tian1997}\\
     & \\
     The graded object $\bigoplus_{\alpha =\hat{1}}^{\hat{\nu}} \mathcal{E}_{-w_{\alpha}}$ of the filtration & Central fibre of a test configuration \\
     & \\
     $k \ge \mathrm{reg} (\mathcal{E})$ in $H^0 (X , \mathcal{E} (k))$ & Exponent of a test configuration \\
      & \\
     Non-Archimedean Donaldson functional (Definition \ref{dfnadfn}) & Non-Archimedean Mabuchi functional (\cite{Paul}, \cite[Section 5]{BHJ2}) \\
      & \\
     $J^{\mathrm{NA}} (\zeta , k)$  in (\ref{eqjnavb})& Non-Archimedean $J$-functional or the minimum norm \cite{BHJ1,Dertwisted} \\
     \hline
  \end{tabular}
  
  \caption{Dictionary} \label{dictionary}
\end{table}
\end{center}

\subsection{Comments on the Deligne pairing} \label{sccmdlp}

While our results can be regarded as a vector bundle version of the results by Boucksom--Hisamoto--Jonsson \cite{BHJ2}, Paul \cite{Paul}, or Phong--Ross--Sturm \cite{PRS}, the proof is not a naive transplantation of the methods used therein. Our method relies on the materials reviewed in Section \ref{revqslimfs}, whereas the Deligne pairing (resp.~the Bott--Chern class) plays a crucially important role in \cite{BHJ2} and \cite{PRS} (resp.~\cite{Paul}).

The method of the Deligne pairing was not used extensively in establishing the results in \cite{HK1,HK2}, unlike in \cite{BHJ2,PRS}. We look at the Donaldson functional from the point of view of the Deligne pairing in this section, but the result we get is not as clear-cut as \cite[Section 1.5]{BHJ2}; the difficulty seems to arise from the fact that there is no explicit formula for the second Bott--Chern class yet (which seems to indicate the difficulty in naively transplanting Paul's argument \cite{Paul} to vector bundles).

In this paper we do not give detailed definitions concerning the Deligne pairing, since we only focus on the very special case; the reader is referred for its proper treatment e.g.~to \cite[Section 1.2]{BHJ2}, \cite{Elkik}, \cite{Moriwaki}, or \cite[Section 2]{PRS} and the references cited therein. The only Deligne pairing that we use in this paper is for the (trivial) flat projective morphism $\pi: X \to \mathrm{pt}$ which maps $X$ to a point. For holomorphic line bundles $L_1 , \dots , L_{n+1}$ we can define the Deligne pairing line bundle $\langle L_1 , \dots , L_{n+1} \rangle_X$ over the point (i.e.~a $\mathbb{C}$-vector space). Given hermitian metrics $\phi_1 , \dots ,  \phi_{n+1}$ on $L_1 , \dots , L_{n+1}$ we can furthermore define a continuous metric $\langle \phi_1 , \dots , \phi_{n+1} \rangle_X$ on $\langle L_1 , \dots , L_{n+1} \rangle_X$. Moreover, the construction is ``functorial'' in the sense as explained in the references cited above. When we give another hermitian metric $\phi_1$ on $L_1$, we have the change of metric formula (see e.g.~\cite[(1.5)]{BHJ2} or \cite[(2.5)]{PRS})
\begin{equation*}
	\langle \phi_1 - \phi'_1 , \dots , \phi_{n+1} \rangle_X = \int_X (\phi_1 - \phi'_1) \eta_{\phi_2} \wedge \cdots \wedge \eta_{\phi_{n+1}}
\end{equation*}
where $\eta_{\phi_i}$ is the curvature form of $\phi_i$ ($i=2 , \dots , n+1$), and the additive notation is used to denote the tensor product of hermitian metrics. In what follows, we shall consider the Deligne pairing for the case $L_2 = \cdots = L_{n+1} = L$. It is well-known that many important functionals that appear in the study of constant scalar curvature K\"ahler or K\"ahler--Einstein metrics can be written as a change of metric formula of an appropriate Deligne pairing line bundle (cf.~\cite[Section 1.5]{BHJ2}). An analogous result holds for the vector bundle case, as stated below.

\begin{prop}
	There exists a $\mathbb{Q}$-line bundle $\mathcal{L}$ on $X$ such that $c_1 (\mathcal{L}) = \Lambda ch_2 (\mathcal{E})$, where $\Lambda$ is the adjoint Lefschetz operator on $H^* (X , \mathbb{C})$, such that the Donaldson functional $\mathcal{M}^{Don}$ can be written as a change of metric formula
	\begin{equation*}
		\langle \psi_1 - \psi_0 , \phi , \dots , \phi \rangle_{X} - \frac{\mu(\mathcal{E})}{\mathrm{Vol}_L} \langle \det h_1 - \det h_0 , \phi , \dots , \phi \rangle_{X}
	\end{equation*}
	 for the Deligne pairing
	\begin{equation*}
		\langle \mathcal{L} , L , \dots , L \rangle_{X} - \frac{\mu(\mathcal{E})}{\mathrm{Vol}_L} \langle \det \mathcal{E} , L , \dots , L \rangle_{X},
	\end{equation*}
	which is a $\mathbb{Q}$-line bundle over a point, where we wrote $\phi$ for the hermitian metric on $L$ whose associated K\"ahler form is $\omega$, and $\psi_1$ (resp.~$\psi_0$) is a certain hermitian metric on $\mathcal{L}$ which depends on $h_1$ (resp.~$h_0$).
\end{prop}

In terms of the Fubini--Study metrics that we discussed earlier, the right family to look at should be $Y := X \times SL (H^0 (X , \mathcal{E} (k))^{\vee})$ with the flat projective morphism $\pi: Y \to SL (H^0 (X , \mathcal{E} (k))^{\vee})$ defined by the second projection, as opposed to the trivial $\pi : X \to \mathrm{pt}$, but we get an analogous result for this case by the functoriality of the Deligne pairing.

\begin{proof}
	It is well-known \cite[Section 1.2]{Don1985} that the Donaldson functional
	\begin{equation*}
	\mathcal{M}^{Don} (h_1 , h_0) = \mathcal{M}^{Don}_1 (h_1 , h_0) - \mu (\mathcal{E}) \mathcal{M}^{Don}_2 (h_1 , h_0)	
	\end{equation*}
	can be written in terms of the Bott--Chern characteristic forms, with
\begin{equation*}
	\mathcal{M}^{Don}_1 (h_1 , h_0) = \int_{X} \mathbf{BC}_2 ( \mathcal{E} , h_1 , h_0) \wedge \frac{\omega^{n-1}}{(n-1)!},
\end{equation*}
and
\begin{equation*}
	\mathcal{M}^{Don}_2 (h_1 , h_0) = \int_{X} \mathbf{BC}_1 ( \mathcal{E} , h_1 , h_0) \wedge \frac{\omega^n}{n!},
\end{equation*}
where the Bott--Chern characteristic forms $\mathbf{BC}_i ( \mathcal{E} , h_1 , h_0)$ ($i=1 , \dots , n$) are a collection of certain secondary characteristic forms, defined modulo $\partial$- and $\bar{\partial}$-exact forms, such that
\begin{equation*}
	- \sqrt{-1} \partial \bar{\partial} \mathbf{BC}_i ( \mathcal{E} , h_1 , h_0) = ch_i (\mathcal{E} , h_1) - ch_i (\mathcal{E} , h_0),
\end{equation*}
where $ch_i (\mathcal{E} , h)$ stands for the $i$-th term of the Chern character form
\begin{equation*}
	ch (\mathcal{E} , h) = \sum_{i=1}^n ch_i (\mathcal{E} , h) = \mathrm{tr} \left( \exp \left( F_h \right) \right)
\end{equation*}
where $F_h$ is $\sqrt{-1}/ 2 \pi$ times the curvature form of $h$. While it is easy to see $\mathbf{BC}_1( \mathcal{E} , h_1 , h_0) = \log \det h_1 h_0^{-1}$, an explicit formula for $\mathbf{BC}_2 ( \mathcal{E}  , h_1 , h_0)$ does not seem to be known yet. See \cite{Don1985,Takh,Tian2000} for more details on the above.

We now write $\Lambda$ for the adjoint Lefschetz operator $\ast^{-1} ( \omega \wedge \cdot ) \ast$ on differential forms on $X$, defined with respect to $\omega$, where $\ast$ is the Hodge star operator with respect to $\omega$. Note also that the adjoint Lefschetz operator on 2-forms equals the metric contraction by $\omega$. We find
\begin{equation*}
	\mathcal{M}^{Don}_1 (h_1 , h_0) = \int_{X} \Lambda \mathbf{BC}_2 ( \mathcal{E} , h_1 , h_0) \frac{\omega^{n}}{n!} .
\end{equation*}
Recalling the well-known K\"ahler identities \cite[Proposition 6.5]{Voisin}
\begin{enumerate}
	\item $[\Lambda , \bar{\partial}] = - \ai \partial^*$,
	\item $[\Lambda , \partial] = \ai \bar{\partial}^*$,
\end{enumerate}
we find, modulo $\partial$- and $\bar{\partial}$-exact forms, that
\begin{align*}
	&\ai \partial \bar{\partial} \Lambda \mathbf{BC}_2 ( \mathcal{E} , h_1 , h_0)\\ 
	&= \ai \left (\Lambda \partial \bar{\partial} - \ai \bar{\partial}^* \bar{\partial} + \ai \partial \partial^* \right)  \mathbf{BC}_2 ( \mathcal{E} , h_1 , h_0) \\
	&= \Lambda (\tr(F_{h_1}^2) - \tr(F_{h_0}^2)) + \Delta_{\bar{\partial}} \mathbf{BC}_2 ( \mathcal{E} , h_1 , h_0)\mod \mathrm{im} \partial + \mathrm{im} \bar{\partial} .
\end{align*}
Note that the second term involving the Laplacian, as well as the $\partial$- and $\bar{\partial}$-exact forms, vanish under the integration.

Now, as in \cite{BHJ2}, we would like to regard $\mathcal{M}^{Don} (h_1 , h_0)$ as a change-of-metric formula for the Deligne pairing. It is easier to deal with $\mathcal{M}^{Don}_2$, since it can be written manifestly as a change of metric formula 
\begin{equation*}
	\frac{1}{\mathrm{Vol}_L} \langle \det h_1 - \det h_0 , \phi , \dots , \phi \rangle_{X}
\end{equation*}
on the line bundle $\frac{1}{\mathrm{Vol}_L} \langle \det \mathcal{E} , L , \dots , L \rangle_{X}$ over the point.

We consider $\mathcal{M}^{Don}_1 (h_1 , h_0)$. Let $\mathcal{L}$ be a holomorphic line bundle on $X$ defined as follows. Recalling that $\Lambda$ induces an operator on the cohomology ring $H^* (X , \mathbb{Z}) / \text{torsion}$, since $[\omega] = c_1 (L)$ is an integral cohomology class, we find that $\Lambda ch_2 (\mathcal{E})$ defines a closed real rational $(1,1)$-form on $X$, which in turn can be realised as the first Chern class of a holomorphic $\mathbb{Q}$-line bundle $\mathcal{L}$ by the Lefschetz $(1,1)$-theorem (see e.g.~\cite[Theorem 7.2]{Voisin}). This is the line bundle that we are after, which is well-defined up to an element of the Picard variety of $X$.

Let $h_{\mathcal{L}}$ be a reference hermitian metric on $\mathcal{L}$. We then define a hermitian metric $e^{-\psi_1} h_{\mathcal{L}}$ so that its curvature form is equal to $\Lambda ch_2 (\mathcal{E} , h_1) =  \Lambda \tr(F_{h_1}^2)$, modulo $\mathrm{im} ( \partial) + \mathrm{im} (\bar{\partial}) + \mathrm{im} (\Delta_{\bar{\partial}} )$. We can always find such $\psi_1$ by solving Laplace's equation, which depends on $h_1$, since $\Lambda \tr(F_{h_1}^2)$ is a de Rham representative of $c_1 (\mathcal{L}) = \Lambda ch_2 (\mathcal{E})$. Likewise we define $\psi_0$ so that the curvature form of $e^{-\psi_0} h_{\mathcal{L}}$ is equal to $\Lambda ch_2 (\mathcal{E} , h_0) =  \Lambda \tr(F_{h_0}^2)$, modulo $\mathrm{im} ( \partial) + \mathrm{im} (\bar{\partial}) + \mathrm{im} (\Delta_{\bar{\partial}} )$.
\end{proof}

As we can see in the proof above, the dependence of $\psi_i$ on $h_i$ ($i=1,2$) is not straightforward; it seems this is partially because no explicit formula is known for the second Bott--Chern class $\mathbf{BC}_2 ( \mathcal{E}  , h_1 , h_0)$.

\begin{rem}
	While it is difficult to explicitly write down $\mathcal{L}$ in terms of $\mathcal{E}$ as pointed out in the above, we can be slightly more specific about it by recalling the Lefschetz decomposition of $H^* (X , \mathbb{C})$ \cite[Theorem 6.4]{Voisin}. This implies that $ch_2 (\mathcal{E})$ can be written uniquely as
\begin{equation*}
	ch_2 (\mathcal{E}) = a L_{\omega}^2 \cdot 1 +  L_{\omega} \cdot \eta_1 +  \eta_2
\end{equation*}
where $a \in \mathbb{Q}$, $L_{\omega}$ is the operator defined by $[ \omega] \wedge \cdot $, and $\eta_1$, $\eta_2$ are primitive forms, i.e.~$\Lambda \eta_i =0$ for $i=1,2$. A well-known result in K\"ahler geometry \cite[Lemma 6.19]{Voisin} says $[L_{\omega} , \Lambda] = (k-n) \mathrm{id}$ on real $k$-forms. Applying this, we have
\begin{align*}
	\Lambda L_{\omega}^2 \cdot 1 &= \left( L_{\omega} (L_{\omega} \Lambda - n) - (2-n) L_{\omega} \right) \cdot 1 = -2L_{\omega} \cdot 1 \\
	\Lambda L_{\omega} \cdot \eta_1 &= -(2-n) \eta_1	
\end{align*}
and hence
\begin{equation*}
	\Lambda ch_2 (\mathcal{E}) = -2 a [\omega] - (2-n) \eta_1,
\end{equation*}
which implies
\begin{equation*}
	c_1 (\mathcal{L}) = -2 a [\omega] - (2-n) \eta_1 .
\end{equation*}
\end{rem}

\section{Bergman 1-parameter subgroups as subgeodesics} \label{scbg1pss}

In the above dictionary (table \ref{dictionary}) we saw that a 1-PS in the Quot-scheme can be regarded as a test configuration for vector bundles. In the case of varieties, it is well-known that test configurations define subgeodesics (see e.g.~\cite[Section 2.4]{Ber16} or \cite[Section 3.1]{BHJ2}). In this section we prove that an analogous result holds for the vector bundle case. We start by recalling \cite[Section 6.2]{Kobook} that the geodesic (with respect to the natural $L^2$-metric on $\mathcal{H}_{\infty}$) is a piecewise $C^1$-family $\{ h_t \}_{t \ge 0}$ of smooth hermitian metrics satisfying
\begin{equation*}
	\partial_t (h^{-1}_{t} \partial_t h_{t}) = 0.
\end{equation*}
Thus, we aim to prove that we have
\begin{equation*}
	\partial_t (h^{-1}_{\sigma_t} \partial_t h_{\sigma_t}) \ge 0
\end{equation*}
for a Bergman 1-PS $\{ \sigma_t \}_{t \ge 0}$. A more precise statement can be found in Proposition \ref{prberggoed}. The proof is given by re-writing $\partial_t (h^{-1}_{\sigma_t} \partial_t h_{\sigma_t})$ in an appropriate manner, which occupies most of what follows.

Let $\sigma_t = e^{\zeta t}$ for a trace-free hermitian matrix $\zeta$, and write $u := \zeta^* + \zeta = 2 \zeta$. Recalling the notation (\ref{dfsmmpq}) and (\ref{dffsmsgm}), we compute $\partial_t (h^{-1}_{\sigma_t} \partial_t h_{\sigma_t})$ as
\begin{align}
	& \partial_t ( (Q^* \sigma_t^* \sigma_t Q)^{-1} (\partial_t Q^* \sigma_t^* \sigma_t Q ))  = \partial_t ( (Q^* \sigma_t^* \sigma_t Q)^{-1}  Q^* \sigma_t^* u \sigma_t Q ) \label{eqsdfssg} \\
	&= (Q^* \sigma_t^* \sigma_t Q)^{-1}Q^* \sigma_t^* u^2 \sigma_t Q - (Q^* \sigma_t^* \sigma_t Q)^{-1}Q^* \sigma_t^* u \sigma_t Q (Q^* \sigma_t^* \sigma_t Q)^{-1} Q^* \sigma_t^* u \sigma_t Q. \notag
\end{align}
Our aim is to simplify this expression. We start with the following lemma.

\begin{lem} \label{hsigcommute}
	Any two of $Q^* \sigma_t^* \sigma_t Q$, $Q^* \sigma_t^* u \sigma_t Q$, $Q^* \sigma_t^* u^2 \sigma_t Q$ pairwise commute.
\end{lem}

In the above statement, $Q^* \sigma_t^* \sigma_t Q$, $Q^* \sigma_t^* u \sigma_t Q$, $Q^* \sigma_t^* u^2 \sigma_t Q$ are regarded as hermitian endomorphisms on $\mathcal{E}$ by fixing a $Q^*Q$-orthonormal frame for $\mathcal{E}$. In what follows, we shall also write $h_{\mathrm{ref}}$ for $Q^*Q$.

\begin{proof}
	We may fix a point $x \in X$ once and for all, and work on the fibre $\mathcal{E}_x$ over $x$. Choosing an orthonormal frame of $\mathcal{E}_x$ with respect to the reference hermitian metric $h_{\mathrm{ref}} = Q^*Q$, we may assume $Q^*Q = I_r$, where $I_r$ is the $r \times r$ identity matrix. Further, by choosing an appropriate basis for $H^0(\mathcal{E}(k))$, we may assume that we can write $Q^* = \begin{pmatrix} I_r & 0 \end{pmatrix}$ and $Q = \begin{pmatrix} I_r \\ 0 \end{pmatrix}$ with respect to a certain decomposition $H^0(\mathcal{E}(k)) = V_r \oplus V_{N-r}$.
		
	Suppose that we have two hermitian matrices $P_1$ and $P_2$, which can be written as $P_1=\begin{pmatrix} A_1 & B_1 \\ C_1 & D_1 \end{pmatrix}$ and $P_2=\begin{pmatrix} A_2 & B_2 \\ C_2 & D_2 \end{pmatrix}$ with respect to the above block decomposition. Then
	\begin{align*}
		Q^* \begin{pmatrix} A_1 & B_1 \\ C_1 & D_1 \end{pmatrix} Q Q^* \begin{pmatrix} A_2 & B_2 \\ C_2 & D_2 \end{pmatrix} Q &= Q^* \begin{pmatrix} A_1 & B_1 \\ C_1 & D_1 \end{pmatrix} \begin{pmatrix} I_r & 0 \\ 0 & 0 \end{pmatrix} \begin{pmatrix} A_2 & B_2 \\ C_2 & D_2 \end{pmatrix} Q\\
		&= \begin{pmatrix} I_r & 0 \end{pmatrix} \begin{pmatrix} A_1A_2 & A_1B_2 \\ C_1A_2 & C_1 B_2 \end{pmatrix} \begin{pmatrix} I_r \\ 0 \end{pmatrix} \\
		&=A_1A_2.
	\end{align*}
	Thus, to prove commutativity of $Q^*P_1 Q$ and $Q^* P_2 Q$, it suffices to show $A_1A_2 = A_2A_1$, which is in turn equivalent to showing that $U_r^* A_1 U_r$ commutes with $U_r^* A_2 U_r$ for some $r \times r$ unitary matrix $U_r$. 
	

Recalling that $\zeta = \zeta^*$ and $\sigma_t = e^{\zeta t}$, we find that $\sigma_t^* \sigma_t$, $\sigma_t^* u \sigma_t$, $\sigma_t^* u^2 \sigma_t$ are all simultaneously diagonalisable, and hence they commute. For the choice of the subspace $V_r \le H^0 (\mathcal{E}(k))$ as above, we have hermitian forms on $V_r$ defined by restriction of $\sigma_t^* \sigma_t$, $\sigma_t^* u \sigma_t$, and $\sigma_t^* u^2 \sigma_t$; these are the matrices denoted by $A_t$, $A'_t$, $A''_t$ in the formulae below:
\begin{equation*}
	\sigma_t^* \sigma_t  = \begin{pmatrix} A_t & B_t \\ C_t & D_t \end{pmatrix}, \quad
		 \sigma_t^* u \sigma_t =  \begin{pmatrix} A'_t & B'_t \\ C'_t & D'_t \end{pmatrix}, \quad
		 \sigma_t^* u^2 \sigma_t  = \begin{pmatrix} A''_t & B''_t \\ C''_t & D''_t \end{pmatrix},
\end{equation*}
where the block decomposition is in terms of $H^0(\mathcal{E}(k)) = V_r \oplus V_{N-r}$. By writing down a basis for $V_r$ in terms of the diagonalising basis for $\zeta$, we find that all the three hermitian forms (or, more precisely, the associated hermitian endomorphisms) $A_t$, $A'_t$, $A''_t$ on $V_r$ thus defined pairwise commute; note that such a basis for $V_r$ can be given by a one that is unitarily equivalent to the one we started with. Thus, combined with the general argument above, for each fixed $x \in X$, we find that all of $Q^* \sigma_t^* \sigma_t Q$, $Q^* \sigma_t^* u \sigma_t Q$, $Q^* \sigma_t^* u^2 \sigma_t Q$ pairwise commute at $\mathcal{E}_x$. Since this holds for all $x \in X$ and $Q^* \sigma_t^* \sigma_t Q$, $Q^* \sigma_t^* u \sigma_t Q$, $Q^* \sigma_t^* u^2 \sigma_t Q$ are tensorial, we conclude the required commutativity among them.
\end{proof}

The following definition, artificial as it may seem, plays an important role.

\begin{definition}
	For a $k$-th Fubini--Study metric $h_{\sigma_t}$ we define $\mathcal{F}(h_{\sigma_t}) \in \mathrm{Hom}_{C^{\infty}_X} (\mathcal{E} , H^0(\mathcal{E} (k)) \otimes C^{\infty}_X )$ as
	\begin{equation*}
		\mathcal{F}(h_{\sigma_t}) = \left( \frac{d \sigma_t}{dt}  Q - \sigma Q \left( h^{-1}_{\sigma_t} \partial_t h_{\sigma_t} \right) \right) h_{\sigma_t}^{-1/2},
	\end{equation*}
	where $h_{\sigma_t}^{-1/2}$ (regarded as a hermitian endomorphism on $\mathcal{E}$) is defined fibrewise with respect to the $h_{\mathrm{ref}}$-orthonormal frame.
\end{definition}

Note that $\mathcal{F}(h_{\sigma_t})$, a fibrewise $N\times r$ matrix varying smoothly in $x$, is a tensorial quantity since $h_{\sigma_t}^{-1/2}$ is tensorial.

\begin{prop} \label{prberggoed}
	For $\mathcal{F}(h_{\sigma_t})$ defined as above, we have
	\begin{equation*}
	\partial_t (h^{-1}_{\sigma_t} \partial_t h_{\sigma_t}) = \mathcal{F}(h_{\sigma_t})^* \mathcal{F}(h_{\sigma_t}) \ge 0,
	\end{equation*}
where $\mathcal{F}(h_{\sigma_t})^*$ is the (fibrewise) conjugate transpose of $\mathcal{F}(h_{\sigma_t})$ with respect to $h_{\mathrm{ref}}$, and the inequality is that of the fibrewise hermitian form.
\end{prop}

\begin{rem}
The weak form of the above proposition $\mathrm{tr} (\partial_t (h^{-1}_{\sigma_t} \partial_t h_{\sigma_t})) \ge 0$ is due to Phong--Sturm \cite[Lemma 2.2]{P-S03} and this was used to prove the convexity of the balancing energy (see Wang \cite[Lemma 3.5]{Wang1} and also Lemma \ref{lcvxm2}).
\end{rem}

The proposition above immediately implies the following result.

\begin{cor} \label{crb1pssgd}
	The Bergman 1-PS define a subgeodesic in the space of hermitian metrics.
\end{cor}


\begin{proof}[Proof of Proposition \ref{prberggoed}]
Fixing an $h_{\mathrm{ref}}$-orthonormal frame to identify hermitian forms and endomorphisms, we compute
\begin{align*}
	&\left( Q^* \sigma_t^* u -(Q^* \sigma_t^* \sigma_t Q)^{-1} (Q^* \sigma_t^* u \sigma_t Q Q^* \sigma_t^*) \right) \\
	&\ \ \ \times \left( u \sigma_t Q -(\sigma_t Q Q^* \sigma_t^* u \sigma_t Q) (Q^* \sigma_t^* \sigma_t Q)^{-1} \right) \\
	&= Q^* \sigma_t^* u^2 \sigma_t Q \\
	&\ \ \ - (Q^* \sigma_t^* u \sigma_t Q)^2 (Q^* \sigma_t^* \sigma_t Q)^{-1} - (Q^* \sigma_t^* \sigma_t Q)^{-1} (Q^* \sigma_t^* u \sigma_t Q)^2 \\
	&\ \ \ +  (Q^* \sigma_t^* \sigma_t Q)^{-1} (Q^* \sigma_t^* u \sigma_t Q) (Q^* \sigma_t^* \sigma_t Q) (Q^* \sigma_t^* u \sigma_t Q) (Q^* \sigma_t^* \sigma_t Q)^{-1} \\
	&= Q^* \sigma_t^* u^2 \sigma_t Q - (Q^* \sigma_t^* u \sigma_t Q) (Q^* \sigma_t^* \sigma_t Q)^{-1} (Q^* \sigma_t^* u \sigma_t Q) ,
\end{align*}
where we used Lemma \ref{hsigcommute} in the last equality. We recall the equation (\ref{eqsdfssg}) and apply $(Q^* \sigma_t^* \sigma_t Q)^{-1}$ from the left to get the claimed result
\begin{align*}
	&\partial_t( h^{-1}_{\sigma_t} \partial_t h_{\sigma_t} ) \\
	&=(Q^* \sigma_t^* \sigma_t Q)^{-1}Q^* \sigma_t^* u^2 \sigma_t Q - (Q^* \sigma_t^* \sigma_t Q)^{-1}Q^* \sigma_t^* u \sigma_t Q (Q^* \sigma_t^* \sigma_t Q)^{-1} Q^* \sigma_t^* u \sigma_t Q \\
	&= (Q^* \sigma_t^* \sigma_t Q)^{-1/2} \left( Q^* \sigma_t^* u -(Q^* \sigma_t^* \sigma_t Q)^{-1} (Q^* \sigma_t^* u \sigma_t Q Q^* \sigma_t^*) \right) \\
	&\ \ \ \times \left( u \sigma_t Q -(\sigma_t Q Q^* \sigma_t^* u \sigma_t Q) (Q^* \sigma_t^* \sigma_t Q)^{-1} \right) (Q^* \sigma_t^* \sigma_t Q)^{-1/2},
\end{align*}
where in the last line we used the fact that $(Q^* \sigma_t^* \sigma_t Q)^{-1}$ is positive definite hermitian, and hence $(Q^* \sigma_t^* \sigma_t Q)^{-1/2}$, defined with respect to the fixed $h_{\mathrm{ref}}$-orthonormal frame, commutes with $Q^* \sigma_t^* \sigma_t Q$, $Q^* \sigma_t^* u \sigma_t Q$, and $Q^* \sigma_t^* u^2 \sigma_t Q$, again by Lemma \ref{hsigcommute}.
\end{proof}

An interesting question is to consider when we have $\mathcal{F} (h_{\sigma_t}) = 0$, for which $\partial_t (h^{-1}_{\sigma_t} \partial_t h_{\sigma_t}) = 0$. Note that $\mathcal{F} (h_{\sigma_t}) = 0$ is equivalent to
\begin{equation*}
	\zeta  Q =  Q \left( h^{-1}_{\sigma_t} \partial_t h_{\sigma_t} \right) 
\end{equation*}
since $\sigma_t$ commutes with $\zeta$. By taking the fibrewise hermitian conjugate, this is further equivalent to
\begin{equation*}
	Q^* \zeta =   \left( h^{-1}_{\sigma_t} \partial_t h_{\sigma_t} \right) Q^*
\end{equation*}
by noting that $\zeta$ and $h_{\sigma_t}$ are both hermitian. This means that the operator $\mathcal{F}$ captures the failure of commutativity of the following diagrams:
\begin{displaymath}
		\xymatrix{ 
		\mathcal{E} \ar@{->}[d]_{Q} \ar@{->}[r]^{h^{-1}_{\sigma_t} \partial_t h_{\sigma_t}} & \mathcal{E} \ar@{->}[d]^{Q} \\
		H^0(\mathcal{E}(k)) \ar@{->}[r]_{\zeta} & H^0(\mathcal{E}(k)) ,
		}
		\quad
		\xymatrix{ 
		\mathcal{E}  \ar@{->}[r]^{h^{-1}_{\sigma_t} \partial_t h_{\sigma_t}} & \mathcal{E}  \\
		\ar@{->}[u]^{Q^*} H^0(\mathcal{E}(k)) \ar@{->}[r]_{\zeta} & H^0(\mathcal{E}(k)) \ar@{->}[u]_{Q^*}.
		}
\end{displaymath}
An example of when it happens is when $\mathcal{E}$ splits in a direct sum of holomorphic vector bundles $\mathcal{E} = \bigoplus_i \mathcal{E}_i$ and $\zeta$ acts as a constant scalar multiplication on each $H^0 (\mathcal{E}_i (k))$. Whether this is the only case when the above diagrams commute may be an interesting problem, but we do not touch on it in this paper.


\begin{rem} \label{rmsgdsgvr}
	The argument above being local, the same result holds over the nonsingular locus $X \setminus ( \mathrm{Sing} (X) \cup \mathrm{Sing} (\mathcal{E}))$ of $X$ even when $X$ is a singular variety and $\mathcal{E}$ is a torsion-free sheaf.
\end{rem}

\section{Gieseker stability and balanced metrics} \label{scgstbm}

In this section, we apply the method of the Quot-scheme limit as surveyed in Section \ref{revqslimfs} to provide a variational characterisation of the Gieseker stability of a torsion-free sheaf on a $\mathbb{Q}$-Gorenstein log terminal projective variety (Theorem \ref{thbalgies}); this generalises the result first proved by X.W.~Wang \cite{Wang1} for holomorphic vector bundles over a smooth projective variety (with an alternative proof given by Phong--Sturm \cite{P-S03}). While this can be seen as an application of materials in Section \ref{revqslimfs} to the singular case, as pointed out in Remark \ref{rmsgvar}, the method of our proof is new even for the regular case considered by \cite{Wang1,P-S03}, in that it does not use the Chow-type norm of the Gieseker point, which was an essential ingredient of the proofs in \cite{Wang1,P-S03}. Instead, in our proof presented below, the inequality for the Gieseker stability appears explicitly as the positivity of the asymptotic slope of the appropriate energy functional (see Proposition \ref{ppasm2gw}).

\begin{rem}
	It is perhaps worth pointing out that a related argument was carried out in the papers by Garc\'ia-Fernandez--Keller--Ross \cite{GF-K-R} and Garc\'ia-Fernandez--Ross \cite{GFRtwH}, for the special case when the filtration is a two-step filtration defined by a saturated subsheaf of $\mathcal{E}$.
\end{rem}

\subsection{Variational formulation of the problem}

We start by recalling the Gieseker stability.

\begin{definition}[Gieseker stability]
 A torsion-free sheaf $\mathcal{E}$ is said to be {\bf Gieseker stable} if the following inequality
 \begin{equation*}
 	\frac{P_{\mathcal{E}}(k)}{\rk(\mathcal{E})} > \frac{P_{\mathcal{F}} (k)}{\rk(\mathcal{F})} \quad \text{ for } k\gg 0,
 \end{equation*}
 holds for all coherent subsheaves $\mathcal{F} \subset \mathcal{E}$ with $0 < \mathrm{rk} (\mathcal{F}) < \rk (\mathcal{E})$, where the Hilbert polynomial $P_{\mathcal{G}} ( k)$ for a coherent sheaf $\mathcal{G}$ on $X$ is defined by $P_{\mathcal{G}} (k) := \sum_{i=0}^n h^i (X , \mathcal{G} (k))$. Gieseker semistability (resp.~polystability) can be defined analogously to the slope semistable (resp.~polystable) case as mentioned in Proposition \ref{prop62hk1}.
\end{definition}

The energy functional that we need to consider in this section is the following, which appeared in \cite{Wang1} and \cite{P-S03}; we first present the version for a holomorphic vector bundle on a smooth variety.

\begin{definition} \label{dfmdon2sm}
	Let $X$ be a smooth projective variety and $\mathcal{E}$ be a holomorphic vector bundle on $X$. The functional $$\mathcal{M}^{Don}_2 : SL(H^0 (X , \mathcal{E} (k))^{\vee}) / SU (N_k) \to \mathbb{R}$$ is defined by
	\begin{equation*}
	\mathcal{M}_2^{Don} (\sigma ) = \frac{1}{\mathrm{Vol}_L} \int_X \log\det( h_{\sigma}h^{-1}_{\mathrm{ref}})\frac{\omega^n}{n!},
\end{equation*}
where $h_{\sigma} = Q^* \sigma^* \sigma Q$ is the Fubini--Study metric defined (as in (\ref{dffsmsgm})) by the hermitian form $\sigma^* \sigma$ on $H^0 (X , \mathcal{E} (k))$ by $\sigma \in SL(H^0 (X , \mathcal{E} (k))^{\vee}) $, and we take the reference metric $h_{\mathrm{ref}}$ to be $Q^*Q$.
\end{definition}

In the above definition and in what follows, we identify an element $\sigma \in SL(H^0 (X , \mathcal{E} (k))^{\vee}) $ with its coset class in $SL(H^0 (X , \mathcal{E} (k))^{\vee}) / SU (N_k)$, noting that $\mathcal{M}_2^{Don} (\sigma )$ depends only on $\sigma^* \sigma$.

\begin{rem}
	Note that $\mathcal{M}_2^{Don} (\sigma_t )$ as defined above is clearly equal to the functional $\mathcal{M}_2^{Don} (h_{\sigma},h_{\mathrm{ref}})$ in Definition \ref{dfdnfc}. We use the above notation in what follows since we only consider the Fubini--Study metrics defined by an element of the coset space $SL(H^0 (X , \mathcal{E} (k))^{\vee}) / SU (N_k)$.
\end{rem}

\begin{rem}
	The method of Quot-scheme limits as discussed in Section \ref{revqslimfs} involves an implicit identification between $\mathcal{E}$ and its dual $\mathcal{E}^{\vee}$ (see \cite[Remark 3.4]{HK1}), which is assumed in what follows, but this does not affect the formulae that appear below as they do not contain the curvature term.
\end{rem}

Suppose now that $X$ is singular and that $\mathcal{E}$ is a torsion-free sheaf on $X$. Instead of the embedding (\ref{dfgrembkdr}), we have a rational map
\begin{equation} \label{dfgrembrt}
	\Phi : X \dashrightarrow \mathrm{Gr} ( \mathrm{rk} (\mathcal{E}) , H^0 (X , \mathcal{E} (k))^{\vee}),
\end{equation}
birational onto its image, by taking $k > \mathrm{reg} (\mathcal{E})$, which is defined on a Zariski open set $X \setminus \mathrm{Sing} (X , \mathcal{E})$, where we defined
\begin{equation*}
	\mathrm{Sing} (X, \mathcal{E}) := \mathrm{Sing} (X) \cup \mathrm{Sing} (\mathcal{E}) .
\end{equation*}
With this understood and writing $h_{\sigma_t}$ for the pullback of the Fubini--Study metric by $\Phi |_{X \setminus  \mathrm{Sing} (X, \mathcal{E}) }$, we can define the functional $\mathcal{M}_2^{Don}$ for the singular case as stated more precisely below.

\begin{definition} \label{dfmdon2svar}
Let $X$ be a normal $\mathbb{Q}$-Gorenstein projective variety with log terminal singularities and that $\mathcal{E}$ is a torsion-free sheaf on $X$. We define the functional $\mathcal{M}^{Don}_2 : SL(H^0 (X , \mathcal{E} (k))^{\vee}) / SU (N_k) \to \mathbb{R}$ by
	\begin{equation*} 
	\mathcal{M}_2^{Don} (\sigma ) = \frac{1}{\mathrm{Vol}} \int_{X \setminus  \mathrm{Sing} (X, \mathcal{E})} \log\det( h_{\sigma}h^{-1}_{\mathrm{ref}}) d V_X ,
\end{equation*}
where $dV_X$ is a volume form on $X \setminus \mathrm{Sing} (X)$ defined by the sections of the pluricanonical bundle $K_X^{\otimes m}$, $m \in \mathbb{N}$ (see e.g.~\cite[(4.35)]{BGKRF} or \cite[Section 6.2]{EGZ09}), which has locally finite mass near $\mathrm{Sing} (X)$ since $X$ is log terminal; this in turn means that 
\begin{equation*}
	\mathrm{Vol} : = \int_{X \setminus  \mathrm{Sing} (X, \mathcal{E}) } dV_X = \int_{X \setminus  \mathrm{Sing} (X) } dV_X < \infty
\end{equation*}
is well-defined.
\end{definition}

We note that the above is indeed a well-defined integral since $X$ is log terminal and $\Phi$ in (\ref{dfgrembrt}) is a rational map. More precisely, first note that by taking a resolution $\pi : \tilde{X} \to X$ we may write
\begin{equation*}
	\mathcal{M}_2^{Don} (\sigma ) = \frac{1}{\mathrm{Vol}} \int_{\tilde{X} \setminus \pi^{-1} ( \mathrm{Sing} (X, \mathcal{E}))} \log \pi^* \det(  h_{\sigma} h^{-1}_{\mathrm{ref}}) \pi^* (d V_X ).
\end{equation*}
We may assume that $\pi^{-1} (\mathrm{Sing} (X) ) = \sum_j E_j$ is a simple normal crossing divisor and that locally in a neighbourhood $U \subset \tilde{X}$, we have $E_j = \{ z_j = 0 \}$ and
\begin{equation*}
	\left. \pi^* (dV_X) \right|_{U \setminus \pi^{-1} (\mathrm{Sing} (X) ) }= \prod_j |z_j|^{2a_j} dV_{U}
\end{equation*}
with $a_j > -1$ for all $j$ (as $X$ is log terminal) and some smooth volume form $dV_U$ on $U$ \cite[Lemma 4.6.5]{BGKRF}. By composing $\pi$ with further blowups, we may assume that $\pi^{-1} (\mathrm{Sing} (\mathcal{E}) ) = \sum_l F_l$ is also a simple normal crossing divisor. Writing $F_l = \{ y_l = 0\}$ locally in a neighbourhood $U \subset \tilde{X}$ and noting that $\pi^* \det(  h_{\sigma} h^{-1}_{\mathrm{ref}})$ has at most poles and zeros of finite order as $\Phi$ is rational, we find that
\begin{align*}
	&\left. \log \pi^* \det(  h_{\sigma} h^{-1}_{\mathrm{ref}}) \pi^* (d V_X ) \right|_{U \setminus \pi^{-1} ( \mathrm{Sing} (X, \mathcal{E})) } \\
	&= \left( \sum_l m_l \log |y_l| + O(1) \right) \prod_j |z_j|^{2a_j} dV_{U}
\end{align*}
with some integers $m_l$ and terms denoted by $O(1)$ that stay bounded over $U$. This is integrable, since $a_j > -1$ for all $j$.

Note also that Definition \ref{dfmdon2svar} is consistent with the one for the smooth varieties (Definition \ref{dfmdon2sm}), since by Yau's theorem \cite{Yau78} we can always find a smooth K\"ahler metric $\omega_{\phi} \in c_1 (L)$ such that $\omega^n_{\phi} = dV_X$, up to rescaling $dV_X$ by a constant. Thus, from now on, without ambiguity we shall adopt Definition \ref{dfmdon2svar} for the definition of $\mathcal{M}_2^{Don}$.
%

A straightforward computation yields that, given a (smooth) path $\{ \sigma_t \}_{t \ge 0} \subset SL(H^0 (X , \mathcal{E} (k))^{\vee})$, we have
\begin{equation*}
	\frac{d^2}{dt^2} \mathcal{M}^{Don}_2 (\sigma_t) =  \frac{1}{\mathrm{Vol}} \int_{X \setminus  \mathrm{Sing} (X, \mathcal{E}) }  \mathrm{tr} (\partial_t (h^{-1}_{\sigma_t} \partial_t h_{\sigma_t})) dV_X .
\end{equation*}
Note that this integral makes sense. As the integrand is invariant under the unitary change of frames of $\mathcal{E}$ (over $X \setminus  \mathrm{Sing} (X, \mathcal{E})$), we may simultaneously diagonalise $h_{\sigma_t}$ and $\partial_t h_{\sigma_t}$ by Lemma \ref{hsigcommute} (regarded as hermitian endomorphisms on $\mathcal{E}$). Since $\partial_t$ does not introduce further poles and does not decrease the order of zeros, we find that $h_{\sigma_t}^{-1} \partial_t h_{\sigma_t}$ (and hence $\partial_t ( h_{\sigma_t}^{-1} \partial_t h_{\sigma_t} )$) is bounded on $X \setminus \mathrm{Sing} (X, \mathcal{E})$ as poles and zeros cancel each other.

The following lemma was first proved for smooth $X$ and locally free $\mathcal{E}$ by X.W.~Wang \cite[Lemma 3.5]{Wang1}, and also by Phong--Sturm \cite[Lemma 2.2]{P-S03}, but we observe that it can be obtained as an immediate consequence of Proposition \ref{prberggoed} and the above formula.

\begin{lem} \label{lcvxm2} 
	$\mathcal{M}^{Don}_2 : SL(H^0 (X , \mathcal{E} (k))^{\vee}) / SU (N_k) \to \mathbb{R}$ is convex along Bergman 1-PS's.
\end{lem}

The above lemma implies that any critical point of $\mathcal{M}^{Don}_2$ is necessarily the global minimum.

The critical point $\sigma \in SL(H^0 (X , \mathcal{E} (k))^{\vee}) / SU (N_k)$, or the associated Fubini--Study metric $h_{\sigma}$, of $\mathcal{M}_2^{Don}$ is called the \textbf{balanced metric}. We can characterise the balanced metric as the one whose Bergman kernel is a constant multiple of the identity, or the one whose associated centre of mass is a constant multiple of the identity, and they can be regarded as providing a finite dimensional approximation of the Hermitian--Einstein metric; we will not discuss these topics here and the reader is referred to \cite{Wang1,Wang2,P-S03} for the details. Note however that these results make sense, at least naively, only when $X$ is smooth and $\mathcal{E}$ is locally free.
%
%

\subsection{Main result and proof}

The main theorem of this section, stated below, is a generalisation of the result of X.W.~Wang \cite{Wang1} (and also Phong--Sturm \cite{P-S03}).

\begin{thm} \label{thbalgies}
	Let $X$ be a normal $\mathbb{Q}$-Gorenstein projective variety with log terminal singularities and that $\mathcal{E}$ is a torsion-free sheaf on $X$. $\mathcal{M}_2^{Don} : SL(H^0 (X , \mathcal{E} (k))^{\vee}) / SU (N_k) \to \mathbb{R}$ admits a critical point for all large enough $k$ if and only if $\mathcal{E}$ is Gieseker stable.
\end{thm}

A novel point of our proof is the following formula for the asymptotic slope of $\mathcal{M}^{Don}_2$ in terms of the invariant that defines the Gieseker stability, which relies on the Quot-scheme limit of Fubini--Study metrics as reviewed in Section \ref{scrmrhk1}.

\begin{prop} \label{ppasm2gw}
	Let $\{ h_{\sigma_t} \}_{t \ge 0}$ be the Bergman 1-PS generated by $\zeta \in \mathfrak{sl} (H^0 (\mathcal{E} (k))^{\vee})$ that has rational eigenvalues. Then we have
	\begin{equation*}
\lim_{t \to + \infty} \frac{\mathcal{M}_2^{Don} ( \sigma_t )}{t} = \frac{2}{j(\zeta , k)} \frac{\mathrm{rk} (\mathcal{E})}{h^0 (X , \mathcal{E}(k))} \sum_{q \in \mathbb{Z}} \mathrm{rk}(\mathcal{E}_{\le q})  \left(  \frac{h^0 (X , \mathcal{E}(k))}{\mathrm{rk} (\mathcal{E})} - \frac{\dim V_{\le q}}{\mathrm{rk} (\mathcal{E}_{ \le q })} \right).
\end{equation*}
\end{prop}

The rest of this section is devoted to the proof of Theorem \ref{thbalgies}. We first prove Proposition \ref{ppasm2gw}, which is a key ingredient in the proof of Theorem \ref{thbalgies}, which in turn follows from Lemmas \ref{lmaslm2} and \ref{lmhlgsw} presented below: Lemma \ref{lmaslm2} is where we critically make use of the renormalised Quot-scheme limit surveyed in Section \ref{revqslimfs}, but Lemma \ref{lmhlgsw} is mostly a repetition of what is well-known to the experts \cite{H-L}. The proof of Theorem \ref{thbalgies} is presented after proving these lemmas. It comes down to proving that the asymptotic slope of $\mathcal{M}_2^{Don}$ is positive if and only if $\mathcal{E}$ is Gieseker stable, but the hypothesis on rational eigenvalues in Proposition \ref{ppasm2gw} presents subtleties that need to be taken care of. This issue is addressed by a slight modification of the argument in \cite{HKJst}.

\begin{lem} \label{lmaslm2}
Let $\{ h_{\sigma_t} \}_{t \ge 0}$ be the Bergman 1-PS generated by a hermitian matrix $\zeta \in \mathfrak{sl} (H^0 (X , \mathcal{E} (k))^{\vee})$, which need not have rational eigenvalues. Then
	\begin{equation*}
\lim_{t \to + \infty} \frac{\mathcal{M}_2^{Don} ( \sigma_t )}{t} = 2 \sum_{\alpha = \hat{1}}^{\hat{\nu}} w_{\alpha} \mathrm{rk} (\mathcal{E}'_{-w_{\alpha}}),
\end{equation*}
where the sheaves $\{ \mathcal{E}'_{-w_{\alpha}} \}_{\alpha = \hat{1}}^{\hat{\nu}}$ are defined by (\ref{satfiltevs}); see also Remark \ref{rmetrlev}.
\end{lem}

\begin{proof}
By using the renormalised metric (\ref{defrn1ps}), we write 
	\begin{equation*}
		 h_{\sigma_t} = e^{ w t} \hat{h}_{\sigma_t} e^{ w t} .
\end{equation*}
Then, we write
\begin{align*}
	&\int_{X \setminus  \mathrm{Sing} (X, \mathcal{E})} \log\det( h_{\sigma_t} h^{-1}_{\mathrm{ref}}) dV_X \\
	&= \int_{X^{reg}} \log \det (e^{2wt}) \det( \hat{h}_{\sigma_t} h^{-1}_{\mathrm{ref}}) dV_X \\
	&=2 t \cdot  \mathrm{Vol} \sum_{\alpha =1}^{\hat{\nu}} w_{\alpha} \mathrm{rk} (\mathcal{E}'_{-w_{\alpha}}) + \int_{X^{reg}} \log \det( \hat{h}_{\sigma_t} h^{-1}_{\mathrm{ref}}) dV_X.
\end{align*}
Note first that the integral $\int_{X^{reg}} \log \det( \hat{h}_{\sigma_t} h^{-1}_{\mathrm{ref}}) dV_X$ is well-defined, by recalling the comments after Definition \ref{dfmdon2svar} and the definition (\ref{defrn1ps}) for $\hat{h}_{\sigma_t}$.

Thus it suffices to prove that the integral $\int_{X^{reg}} \log \det( \hat{h}_{\sigma_t} h^{-1}_{\mathrm{ref}}) dV_X$ remains bounded as $t \to + \infty$. The argument is similar to the proof of \cite[Proposition 3.8]{HK1}, except for that the case under consideration here is much easier. The key ingredient is that $\hat{h}$ degenerates only on a Zariski closed subset $X \setminus X^{reg}$, and that $\hat{h}$ has at worst zeros and poles of finite order by \cite[Lemma 1.22]{HK1} which follows from the fact that the quotient map $\rho$ in \cite[(1.1)]{HK1} is algebraic and that $\Phi$ in (\ref{dfgrembrt}) is rational. Since $\hat{h}$ is well-defined as a hermitian metric over $X^{reg}$ \cite[Lemmas 2.11 and 2.12]{HK1}, we find that the integral on the right hand side remains bounded as $t \to + \infty$, which gives the claimed result.
\end{proof}

\begin{lem} \emph{(cf.~\cite[Section 4.A]{H-L})} \label{lmhlgsw}
Suppose that we have $\zeta \in \mathfrak{sl} (H^0(X, \mathcal{E}(k))^{\vee})$ that gives rise to the filtrations (\ref{VT'}), (\ref{filtevs}), and (\ref{satfiltevs}) by taking the saturation. Then
\begin{equation*}
\sum_{\alpha =1}^{\hat{\nu}} w_{\alpha} \mathrm{rk} (\mathcal{E}'_{-w_{\alpha}}) =\frac{2}{j(\zeta , k)} \frac{\mathrm{rk} (\mathcal{E})}{h^0 (X , \mathcal{E}(k))} \sum_{q \in \mathbb{Z}} \mathrm{rk}(\mathcal{E}_{\le q})  \left(  \frac{ h^0 (X , \mathcal{E}(k))}{\mathrm{rk} (\mathcal{E})} - \frac{\dim V_{\le q}}{\mathrm{rk} (\mathcal{E}_{ \le q })} \right).
\end{equation*}
\end{lem}

\begin{proof}
	Recall the definition (\ref{defsmjzk}) and that we have an integrally graded filtration (\ref{itsatfiltevs}) of $\mathcal{E}$ by subsheaves. We then observe
	\begin{equation*}
		\sum_{\alpha = \hat{1}}^{\hat{\nu}} w_{\alpha} \mathrm{rk} (\mathcal{E}'_{-w_{\alpha}}) =  \frac{1}{j(\zeta , k)} \sum_{\alpha = \hat{1}}^{\hat{\nu}} \bar{w}_{\alpha} \mathrm{rk} (\mathcal{E}'_{-\bar{w}_{\alpha}}) =  \frac{1}{j(\zeta , k)} \sum_{i =1}^{\nu} \bar{w}_{i} \mathrm{rk} (\mathcal{E}'_{-\bar{w}_{i}}).
	\end{equation*}
	by recalling the definition (\ref{defsubswgt}) of $\hat{1} , \dots , \hat{\nu}$, which implies $\mathrm{rk} (\mathcal{E}'_{-w_{i}}) =0$ if and only if $i \not\in \{ \hat{1} , \dots , \hat{\nu} \}$. Note further that
	\begin{equation*}
		\sum_{i =1}^{\nu} \bar{w}_{i} \mathrm{rk} (\mathcal{E}'_{-\bar{w}_{i}}) = - \sum_{q \in \mathbb{Z}} q \cdot \mathrm{rk} (\mathcal{E}'_{q}).
	\end{equation*}
	
	We now perform the calculation that is identical to the one carried out in \cite[Section 4.A]{H-L}: since $\zeta \in \mathfrak{sl} (H^0(X , \mathcal{E}(k))^{\vee})$, we have
	\begin{align*}
		\dim V \sum_{q \in \mathbb{Z}} q \cdot \mathrm{rk} (\mathcal{E}'_{q}) &= \sum_{q \in \mathbb{Z}} q \cdot \left( \mathrm{rk} (\mathcal{E}'_{q}) \cdot \dim V - \mathrm{rk} (\mathcal{E}) \cdot \dim V_{q} \right) \\
		&=- \sum_{q \in \mathbb{Z}} \left( \mathrm{rk} (\mathcal{E}'_{\le q}) \cdot \dim V - \mathrm{rk} (\mathcal{E}) \cdot \dim V_{\le q} \right).
	\end{align*}
	Combining these equalities we get
	\begin{equation*}
		\sum_{\alpha =1}^{\hat{\nu}} w_{\alpha} \mathrm{rk} (\mathcal{E}'_{-w_{\alpha}})  = \frac{1}{j(\zeta , k)} \frac{1}{\dim V} \sum_{q \in \mathbb{Z}} \left( \mathrm{rk} (\mathcal{E}'_{\le q}) \cdot \dim V - \mathrm{rk} (\mathcal{E}) \cdot \dim V_{\le q} \right).
	\end{equation*}
	By substituting in $\dim V = h^0 (X , \mathcal{E} (k))$ and tidying up the terms, we get the desired result.
\end{proof}

\begin{proof}[Proof of Theorem \ref{thbalgies}]
	We first recall \cite[Theorem 4.4.1]{H-L} which is credited to Le Potier in \cite{H-L}, where we note that the multiplicity of a sheaf is just the rank since $X$ is a projective variety \cite[page 11]{H-L}. It states that for all sufficiently large integers $k$, $\mathcal{E}$ being Gieseker stable is equivalent to
	\begin{equation*}
		\frac{h^0 (X , \mathcal{F} (k))}{\mathrm{rk} (\mathcal{F})} < \frac{P_{\mathcal{E}(k)}}{\mathrm{rk} (\mathcal{E})} = \frac{h^0 (X , \mathcal{E} (k))}{\mathrm{rk} (\mathcal{E})}
	\end{equation*}
	for all subsheaves $\mathcal{F} \subset \mathcal{E}$ of rank $0 < \mathrm{rk} (\mathcal{F}) < \mathrm{rk} (\mathcal{E})$. This proves that $\mathcal{E}$ is Gieseker stable if and only if for all sufficiently large $k$ and all $\zeta \in \mathfrak{sl} (H^0 (X , \mathcal{E} (k))^{\vee})$ with rational eigenvalues we have, for the 1-PS $\{ \sigma_t \}_{t \ge 0}$ defined by $\sigma_t = e^{\zeta t}$,
	\begin{equation*}
		\lim_{t \to + \infty} \frac{\mathcal{M}_2^{Don} ( \sigma_t )}{t} > 0.
	\end{equation*}
	
	On the other hand, the convexity of $\mathcal{M}_2^{Don}$ along Bergman 1-PS (Lemma \ref{lcvxm2}) implies that $\mathcal{M}_2^{Don}$ has a critical point if and only if the above inequality holds for all $\zeta \in \mathfrak{sl} (H^0 (X , \mathcal{E} (k))^{\vee})$ that is hermitian but \textit{not} necessarily having rational eigenvalues. We follow the argument in \cite[Lemmas 3.15 and 3.17]{HKJst} to prove that the Gieseker stability in fact implies this seemingly stronger condition. First note that by continuity and Lemma \ref{lmaslm2} we have
	\begin{equation} \label{eqslm2rl}
		\lim_{t \to + \infty} \frac{\mathcal{M}_2^{Don} ( \sigma_t )}{t} = 2 \sum_{\alpha =1}^{\hat{\nu}} w_{\alpha} \mathrm{rk} (\mathcal{E}'_{-w_{\alpha}}) \ge 0
	\end{equation}
	for all $\zeta \in \mathfrak{sl} (H^0 (\mathcal{E} (k))^{\vee})$, not necessarily having rational eigenvalues, by recalling that $\mathcal{E}_{- w_i}$ is well-defined even when $w_i$ is not rational (Remark \ref{rmetrlev}); in the above we wrote $w_1 , \cdots, w_{\nu} \in \mathbb{R}$ for the eigenvalues of $\zeta$. Now, we choose a hermitian matrix $\tilde{\zeta}$ with rational eigenvalues, say $\tilde{w}_1 , \dots , \tilde{w}_{\nu} \in \mathbb{Q}$, so that
	\begin{enumerate}
		\item $V_{-w_i} = V_{-\tilde{w}_i}$ for all $i=1 , \dots , \nu$,
		\item $\tilde{w}_1 > \tilde{w}_2 > \cdots > \tilde{w}_{\nu}$,
		\item $w_1- \tilde{w}_1 > w_2 - \tilde{w}_2 > \cdots > w_{\nu} - \tilde{w}_{\nu}$,
	\end{enumerate}
	which is possible since $\mathbb{Q}$ is dense in $\mathbb{R}$. Note also that the first item above implies that we have $\mathcal{E}_{- w_i} = \mathcal{E}_{- \tilde{w}_i}$ for all $i = 1 , \dots , \nu$. We can then re-write (\ref{eqslm2rl}) as
\begin{equation*}
\lim_{t \to + \infty} \frac{\mathcal{M}_2^{Don} ( \sigma_t )}{t} = 2 \sum_{\alpha =1}^{\hat{\nu}} \tilde{w}_{\alpha} \mathrm{rk} (\mathcal{E}'_{-\tilde{w}_{\alpha}}) + 2 \sum_{\alpha =1}^{\hat{\nu}}( w_{\alpha} - \tilde{w}_{\alpha}) \mathrm{rk} (\mathcal{E}'_{-w_{\alpha}}).
\end{equation*}
The first term on the right hand side is strictly positive since $\tilde{\zeta}$ has rational eigenvalues, and we claim that the second term is nonnegative. This is a consequence of the inequality (\ref{eqslm2rl}) and the following observation: the second term is equal to the asymptotic slope $\lim_{t \to + \infty} \mathcal{M}_2^{Don} (\eta_t ) / t$, where $\{ \eta_t \}_{t \ge 0}$ is the 1-PS defined by $\eta_t = \exp ((\zeta - \tilde{\zeta})t)$ in $SL(H^0 (X , \mathcal{E} (k))^{\vee})$, by observing $V_{-(w_i - \tilde{w}_i)} = V_{-\tilde{w}_i}$ and $\mathcal{E}_{- (w_i - \tilde{w}_i)} = \mathcal{E}_{- \tilde{w}_i}$ for all $i = 1 , \dots , \nu$. Thus, the sum of these two terms is strictly positive, finally implying
\begin{equation*}
	\lim_{t \to + \infty} \frac{\mathcal{M}_2^{Don} ( \sigma_t )}{t} >0
\end{equation*}
for all hermitian $\zeta \in \mathfrak{sl} (H^0 (\mathcal{E} (k))^{\vee})$, as required.

Thus, identifying $t >0$ with the radial direction in $\mathfrak{sl} (H^0(X , \mathcal{E} (k))^{\vee})$ by recalling $\Vert \zeta \Vert_{op} \le 1$ (Remark \ref{rmopnml1}), and also noting that we have
\begin{equation*}
	\mathfrak{sl} (H^0(X , \mathcal{E} (k))^{\vee}) = \mathfrak{su} (N_k) \oplus \ai \mathfrak{su} (N_k)
\end{equation*}
and that $\ai \mathfrak{su} (N_k)$ is the set of hermitian matrices, the geodesic convexity of $\mathcal{M}_2^{Don} ( \sigma_t )$ (Lemma \ref{lcvxm2}) implies that the map
\begin{equation*}
	SL(H^0 (X , \mathcal{E} (k))^{\vee}) / U(N_k) \isom \ai \mathfrak{su} (N_k) \ni \zeta t \mapsto \mathcal{M}_2^{Don} ( \sigma_t ) \in \mathbb{R}
\end{equation*}
is bounded below and proper, where the first arrow is the diffeomorphism given by the global Cartan decomposition. Hence we finally conclude that there exists $\tilde{\sigma} \in SL(H^0 (X , \mathcal{E} (k))^{\vee}) / U(N_k)$ that attains the global minimum of $\mathcal{M}_2^{Don}$.
\end{proof}

\section{Towards effective results and an algorithm for computing Hermitian--Einstein metrics \label{Section-effective}}

	One can notice that in Proposition \ref{prop62hk1}, in order to prove slope stability of the bundle, it is sufficient to check $\mathcal{M}^{\mathrm{NA}} (\zeta , k_0) > 0$ for $\zeta \in \mathfrak{sl} (H^0 (X , \mathcal{E}(k_0))^{\vee})$ with \textit{some} $k_0 \in \mathbb{N}$ satisfying 
	\begin{equation*}
		k _0 \ge \max \{ \mathrm{reg} (\mathcal{E}) , \mathrm{reg} ( \mathcal{F}_{\mathrm{max}} ) \} ,
	\end{equation*}
	where $\mathcal{F}_{\mathrm{max}} \subset \mathcal{E}$ is the maximally destabilising subsheaf; recall that given a holomorphic vector bundle $\mathcal{E}$, there exists a unique maximal destabilising saturated subsheaf $\mathcal{F}_{\mathrm{max}}$ for $\mathcal{E}$ which satisfies
\begin{itemize}
 \item if $\mathcal{F}\subset \mathcal{E}$ is a proper subsheaf of $\mathcal{E}$, then $\mu(\mathcal{F})\leq \mu(\mathcal{F}_{\mathrm{max}} )$;
 \item if $\mu(\mathcal{F})= \mu(\mathcal{F}_{\mathrm{max}})$, then $\rk(\mathcal{F})\leq \rk(\mathcal{F}_{\mathrm{max}})$.
\end{itemize}
For the existence and uniqueness of the maximal destabilising subsheaf the reader is referred to \cite[Lemma 1.3.5]{H-L} and \cite[section V.7, Lemma 7.17]{Kobook}. Note that $\mathcal{F}_{max}$ is slope semistable. With above notations, we have the following proposition.


\begin{prop}\label{res1} Let $\mathcal{E}$ a holomorphic vector bundle over a polarised manifold. 
Set $$k_0=\max(\mathrm{reg}(\mathcal{E}),\mathrm{reg}(\mathcal{F}_{max})).$$
The following assertions are equivalent:
\begin{enumerate}
 \item[(1)] $\mathcal{E}$ is slope  stable,
 \item[(2)] The Donaldson functional ${\mathcal M}^{Don}$ is coercive along  rational Bergman 1-PS at level $k_0$,
  \item[(3)] The Donaldson functional ${\mathcal M}^{Don}$ is coercive along  rational Bergman 1-PS at level $k$ for any $k\geq k_0$.
\end{enumerate}
The following assertions are equivalent:
\begin{enumerate}
 \item[(1')] $\mathcal{E}$ is slope  semistable,
 \item[(2')] The Donaldson functional ${\mathcal M}^{Don}$ along  rational Bergman 1-PS at level $k_0$ is bounded from below,
  \item[(3')] The Donaldson functional ${\mathcal M}^{Don}$ along  rational Bergman 1-PS at level $k$ is bounded from below for any $k\geq k_0$.
\end{enumerate}
\end{prop}
\begin{proof}
 If $\mathcal{E}$ is slope  stable (resp. slope semistable), then one can apply Theorem \ref{thmlnsdf} since $\mathcal{E}$ is $k$-regular for any $k\geq \mathrm{reg}(\mathcal{E})$, see \cite[Lemma 1.7.2]{H-L}. 
 This gives $(1)\Rightarrow (3)$ and obviously $(3)\Rightarrow (2)$ (resp. $(1')\Rightarrow (3')\Rightarrow (2')$).\\
 A special case of our study is given when one is considering a 2-step filtration associated to a regular saturated subsheaf $\mathcal{F} \subset \mathcal{E}$ as described in \cite[Proposition 5.2]{HK1}. For  $\zeta_{\mathcal{F}} \in\mathfrak{sl} (H^0 (\mathcal{E}(k))^{\vee})$ the element defining this filtration, with weights well chosen (cf. \cite[Section 5]{HK1}), one can consider the Bergman 1-PS $\{ h_{\sigma_t} \}_{t \ge 0}$ emanating from $h_k$ and induced by $\zeta_{\mathcal{F}}$ as above. Then we proved the existence of a constant $c_k=c(h_k ,  k) >0$ such that
\begin{equation}\label{belowineq}
	\mathcal{M}^{Don} (h_{\sigma_t} , h_{k}) \ge \rk (\mathcal{F}) (\mu (\mathcal{E}) - \mu( \mathcal{F})) \cdot 2 t - c_k
\end{equation}
for all $t \ge 0$, and a constant $c'_k=c(h_k , \zeta_{\mathcal{F}} , k) >0$ such that
\begin{equation}\label{aboveineq}
	\mathcal{M}^{Don} (h_{\sigma_t} , h_{k}) \le \rk (\mathcal{F}) (\mu (\mathcal{E}) - \mu( \mathcal{F})) \cdot 2 t + c'_k
\end{equation}
holds for all sufficiently large $t > 0$. If we have coercivity at level $k_0$, we apply \eqref{aboveineq} to $\mathcal{F}_{max} $. Since  ${\mathcal M}^{Don}$ along $\{h_{\sigma_t}\}$ grows to $+\infty$ when $t\to +\infty$, this provides $\mu(\mathcal{E})>\mu(\mathcal{F}_{max} )$ and by definition of $\mathcal{F}_{max} $, $\mathcal{E}$ is actually slope  stable. Thus $(2)\Rightarrow (1)$. If we have boundedness from below, we apply both \eqref{aboveineq} and \eqref{belowineq} when $t\to +\infty$ to conclude that $\mu(\mathcal{E})=\mu(\mathcal{F}_{max} )$. This shows $(2')\Rightarrow (1')$.
\end{proof}

\begin{rem}
One could derive a less precise version of above result by just invoking Grothendieck boundedness result \cite[Lemme 2.5]{Grothendieck1960-1961}, \cite[Lemma 1.7.9]{H-L} instead of maximally desabilising subsheaves.  Grothendieck boundedness result  ensures that the family of torsion-free quotients of the vector bundle $\mathcal{E}$ with slope bounded from above is actually a bounded family. Thus the family of saturated coherent subsheaves of $\mathcal{E}$ with slope bounded from below  is also bounded; observe that if $\{ \mathcal{E}/ F_t \}_t$ is a bounded family of quotient sheaves of $\mathcal{E}$, then the family $\{ F_t \}_t$ of subsheaves must be also bounded. Using Grothendieck boundedness result, we get a uniform $k_0$ as in the statement of the theorem which is not explicit (see also e.g.~\cite[Lemma 5.7.16]{Kobook} or \cite[Lemma 2]{Shatz}).
\end{rem}

\begin{rem}
 The proof shows that the slope  stability actually implies coercivity of the Donaldson function on the Bergman space at level $\mathrm{reg}(\mathcal{E})\leq k_0$. We don't expect the converse to be true.
\end{rem}

The regularity of coherent sheaves has been studied since decades and bounds on the regularity have been made more or less explicit. Let's provide some details. If  $\iota: X \to \mathbb{CP}^N$ is a holomorphic embedding and $\mathcal{F}$ a coherent sheaf on $X$ then by the projection formula $\mathrm{reg}(\mathcal{F})=\mathrm{reg}(\iota_*\mathcal{F})$. Using this argument, one can obtain information on the regularity of sheaves by restricting to the projective case. 
From the fundamental work of Mumford, it is known that for any coherent sheaf $\mathcal{F}$ on $\mathbb{CP}^N$, which is isomorphic to a subsheaf of $\bigoplus_{j=1}^{N_0} \mathcal{O}_{\mathbb{CP}^N}$, with Hilbert polynomial $$\chi(\mathcal{F}(k))=\sum_{i=0}^N a_i \binom{k}{i}$$ ($a_i\in \mathbb{Z}$), one has the $k_0$-regularity of $\mathcal{F}$ for $k_0=F(a_0,..,a_N)$ where $F$ is a universal polynomial in $N+1$ variables that depends on $(N,N_0)$ that can be made explicit. For instance, the case of semistable bundles over $\mathbb{CP}^2$ is studied in \cite[Corollary 5.5]{Langer}, see also \cite{EF80}. 
The interest of making effective the Castelnuovo--Mumford regularity $k_0$ in Proposition \ref{res1} becomes clear when one is considering numerical applications. In \cite{Seyyedali} it is presented an algorithm based on ideas of S. Donaldson \cite{Do-2009} to compute balanced metrics (cf. section \ref{scgstbm}) in the set of $k$-th Fubini--Study metrics $\mathcal{B}_k$ on a stable bundle $\mathcal{E}$ that approximate the Hermitian--Einstein metric living on the bundle when $k\to +\infty$. Nevertheless, with the notion of balanced metrics, it remains unclear which minimal $k$ can be chosen to run the algorithm, see for instance \cite{DKLR}.   
If $\mathcal{E}$ is a stable bundle, Proposition \ref{res1} ensures that the Donaldson functional is coercive on the Bergman space $\mathcal{B}_{k_0}$ which has finite dimension, and thus it attains a minimum, say at the metric $h_k^{\min}$. If one denotes $h_{HE}$ the Hermitian--Einstein metric on $\mathcal{E}$, $h\mapsto {\mathcal M}^{Don}(h,h_{HE})$ reaches its minimum at $h=h_{HE}$ where it vanishes. Technically, in order to find  $h_k^{\min}$, one can apply Levenberg--Marquardt algorithm to $h\mapsto \vert {\mathcal M}^{Don}(h,h_{HE})\vert$ restricted to $\mathcal{B}_{k_0}$. By density of the Bergman spaces (cf. \cite[Theorem 1.16]{HK1}), 
$${\mathcal M}^{Don}(h_k^{\min},h_{HE}) \leq \frac{C_{HE}}{k},$$
where $C_{HE}$ is a constant that depends only on the Hermitian--Einstein metric and its covariant derivatives.
Moreover, once the minimum is achieved, it is possible to estimate how far is $h_k^{min}$ from $h_{HE}$ using \cite[Theorem B.2]{HK2}. Actually, in this view, one can introduce $$\delta:=\inf_{x\in X} \frac{\lambda_{\min} }{\lambda_{\max}}$$ where $\lambda_{\max}, \lambda_{\min}$ are the maximum and minimum eigenvalues of $h_k^{\min}h_{HE}^{-1}:=e^{v}$. The problems turns out to measure $1-\delta$ when this quantity is small. But 
the proof of \cite[Theorem B.2]{HK2} shows that we have the inequality
\begin{align*}
 {\mathcal M}^{Don}(h_k^{min},h_{HE})\geq \frac{\delta-1-\log(\delta)}{\log(\delta)^2} C_{\nabla^*\bar{\partial}}^{-1} \Vert v- \bar{v}\Vert^2_{L^2},
\end{align*}
where $\bar{v}=\frac{1}{r \mathrm{Vol}_L} \int_X \tr (v) \frac{\omega^n}{n!} \cdot \Id_{\mathcal{E}} $ is the average of $v$, and 
$C_{\nabla^*\bar{\partial}}$ can be interpreted as the first non zero eigenvalue of the operator $\sqrt{-1}\Lambda \bar{\partial}\partial$ acting on endomorphisms of $E$ and that depends on the metric $h_{HE}$. Denoting $r$ the rank of $\mathcal{E}$,
one has $\Vert v- \bar{v}\Vert^2_{L^2} = \Vert v\Vert^2_{L^2} - \Vert  \bar{v}\Vert^2_{L^2}\geq \frac{1}{r}(\log \delta)^2$
and consequently, 
\begin{align*}
 {\mathcal M}^{Don}(h_k^{min},h_{HE})\geq \frac{1}{r}(\delta-1-\log(\delta) ) C_{\nabla^*\bar{\partial}}^{-1} \sim C_{\nabla^*\bar{\partial}}^{-1} \frac{(1-\delta)^2}{2r} .
\end{align*}

\bibliography{notes.bib}

\end{document}